\documentclass{svproc}
\usepackage{url}

\usepackage{amsmath,amssymb}
\usepackage{mathtools}
\usepackage{tabularx,array,threeparttable}
\usepackage{multirow}
\usepackage{float}
\usepackage[pdftex]{graphicx}
\usepackage{wrapfig}
\usepackage[small]{caption}
\usepackage{graphicx}
\usepackage{subcaption}
\usepackage{tikz}
\usetikzlibrary{automata,arrows,positioning,calc} 
\usepackage{multicol}
\usepackage{enumitem}
\usepackage{pgfplots}
\usepgfplotslibrary{fillbetween}	
\usepackage{xcolor}
\usepackage{colortbl}
\usepackage{cite}
\usepackage{booktabs}
\usepackage{hyperref}

\graphicspath{./figures}
\DeclareGraphicsExtensions{.pdf,.jpeg,.png,.bmp}

\newcommand{\rhocrit}{\rho^{\text{crit}}}
\newcommand{\rhomax}{\rho^{\text{max}}}
\newcommand{\zmax}{z^{\text{max}}}

\newcommand{\nidcs}{{[}n{]}}

\newcommand{\midcs}{{[}m{]}}
\newcommand{\mprimeidcs}{{[}m'{]}}

\newcommand{\barx}{\bar{x}}
\newcommand{\bary}{\bar{y}}
\newcommand{\bx}{\boldsymbol{x}}
\newcommand{\by}{\boldsymbol{y}}

\newcommand{\bs}{\boldsymbol{s}}

\newcommand{\meq}{{m_\mathrm{EQ}}}
\newcommand{\meqidcs}{{{[}m_\mathrm{EQ}{]}}}
\newcommand{\mstareq}{{m^{*}_\mathrm{EQ}}}
\newcommand{\mstareqidcs}{{{[}m^{*}_\text{EQ}}{]}}
\newcommand{\mall}{{m_{\mathrm{ALL}}}}
\newcommand{\mstarall}{{m^{*}_{\mathrm{ALL}}}}
\newcommand{\mallidcs}{{{[}m_{\mathrm{ALL}}{]}}}
\newcommand{\altruismArg}{\kappa}
\newcommand{\altruismUnif}{\kappa_0}
\newcommand{\altruismFn}{\varphi}
\newcommand{\altruismLevSet}{K}
\newcommand{\eqDelay}{\hat{\ell}_0}

\newcommand{\argmin}{\operatornamewithlimits{arg\,min}}
\newcommand{\argmax}{\operatornamewithlimits{arg\,max}}

\newcommand{\ignore}[1]{}

\newcommand{\smallspacingsection}[1]{%
	\vspace*{-12px}\section{#1}\vspace*{-8px}%
}

\begin{document}
\mainmatter              
\title{Altruistic Autonomy: \\Beating Congestion on Shared Roads}
\titlerunning{ }
\author{Erdem B\i y\i k${}^\star$\inst{1}, Daniel A. Lazar${}^\star$\inst{2},
Ramtin Pedarsani\inst{2}, Dorsa Sadigh\inst{3} }
\authorrunning{ }
%
\tocauthor{Erdem B\i y\i k, Daniel A. Lazar, Ramtin Pedarsani, Dorsa Sadigh}
\institute{Electrical Engineering, Stanford University,
\and
Electrical and Computer Engineering, University of California, Santa Barbara,
\and
Computer Science \& Electrical Engineering, Stanford University. \\
\email{ebiyik@stanford.edu, \{dlazar, ramtin\}@ece.ucsb.edu, dorsa@cs.stanford.edu}
\footnotetext{${}^\star$ First two authors contributed equally and are listed in alphabetical order.} 
}
\maketitle              

\begin{abstract}
Traffic congestion has large economic and social costs. The introduction of autonomous vehicles can potentially reduce this congestion, both by increasing network throughput and by enabling a social planner to incentivize users of autonomous vehicles to take longer routes that can alleviate congestion on more direct roads. We formalize the effects of \emph{altruistic autonomy} on roads shared between human drivers and autonomous vehicles.
In this work, we develop a formal model of road congestion on shared roads based on the fundamental diagram of traffic. We consider a network of parallel roads and provide algorithms that compute optimal equilibria that are robust to additional unforeseen demand. We further plan for optimal routings when users have varying degrees of altruism. We find that even with arbitrarily small altruism, total latency can be unboundedly better than without altruism, and that the best selfish equilibrium can be unboundedly better than the worst selfish equilibrium. We validate our theoretical results through microscopic traffic simulations and show average latency decrease of a factor of 4 from worst-case selfish equilibrium to the optimal equilibrium when autonomous vehicles are altruistic.
\end{abstract}

\smallspacingsection{Introduction}
Autonomous and connected vehicles are soon becoming a significant part of roads normally used by human drivers. Such vehicles hold the promise of safer streets, better fuel efficiency, more flexibility in tailoring to specific drivers' needs, and time savings. In addition to benefits impacting individual users, autonomous cars can significantly influence the flow of traffic networks. A coalition of autonomous cars can potentially decrease road congestion through a number of creative techniques. Autonomous vehicles can increase road capacity and throughput, as well as decrease fuel consumption, by forming platoons and maintaining shorter headways to the vehicles they follow~\cite{askari2017effect, lioris2017platoons, motie2016throughput, adler2016optimal}. They can smooth shockwaves and instabilities that form in congested flows of human-driven vehicles~\cite{wu2017emergent, wu2018stabilizing, cui2017stabilizing, stern2018dissipation}. Finally, autonomous vehicles can incentivize people to take \emph{unselfish} routes, i.e., routes leading to a lower overall traffic delay for all drivers on the network. 
Previous work has studied methods such as tolling and incentives to influence drivers' routing choices by, e.g., shifting their commute times or incentivizing them to use less congested roads or less congesting modes of transportation~\cite{merugu2009incentive, pluntke2013insinc, nie2013managing}.
Our key insight is how people, who are incentivized to behave altruistically, along with vehicles with autonomous capabilities, can be leveraged to positively influence traffic networks.
Our key contribution is to formally study and plan for \emph{altruistic} autonomous cars that can influence traffic flows in order to achieve lower latencies. The setting we are interested in is as follows: assume there is a population of selfish human drivers and a population of autonomous car users, some of whom are willing to take routes with varying degrees of inefficiency. \emph{How can we best utilize this altruistic autonomy to minimize average user delay?}

To answer this, we first develop a model of vehicle flow that reflects changes in road capacity due to efficient platooning of autonomous cars on shared roads. The latency experienced by users of a road is a function of the total vehicle flow as well as the \emph{autonomy level}, the fraction of vehicles that are autonomous. We use the terminology \emph{latency} to refer to the delay experienced by a driver from taking a road, \emph{flow} to refer to the vehicle flux on a road (vehicles per second), and \emph{density} to describe the number of cars per meter on a road. Our mixed-autonomy model is developed based on the commonly used fundamental diagram of traffic~\cite{DAGANZO1994269}, which relates vehicle density to vehicle flow on roads with only human-driven cars.

Using our mixed-autonomy traffic flow model for a network of parallel roads, we study selfish equilibria in which human-drivers and users of autonomous vehicles all pick their routes to minimize their individual delays. As it may be difficult to exactly gauge flow demands, we define a notion of \emph{robustness} that quantifies how resilient an equilibrium is to additional unforeseen demand. We establish properties of these equilibria and provide a polynomial-time algorithm for finding a robust equilibrium that minimizes overall delay.

We then address how to best use autonomous vehicles that are altruistic. We develop the notion of an \emph{altruism profile}, which reflects the varying degrees to which users of autonomous vehicles are willing to take routes longer than the quickest route. We establish properties of optimal altruistic equilibria and provide a polynomial-time algorithm for finding such routings. 

We validate our theoretical results through experiments in a microscopic traffic simulator. We show that an optimal selfish equilibrium can improve latency by a factor of two, and utilizing altruism can improve latency by another factor of two in a realistic driving scenario.
Our contributions in this work are as follows:
\begin{itemize}[nosep]
\item We develop a formal model of traffic flow on mixed-autonomy roads.
\item We design an optimization-based algorithm to find a best-case Nash equilibrium.
\item We define a robustness measure for the ability to allocate additional flow demand and develop an algorithm to optimize for this measure.
\item We develop an algorithm to minimize total experienced latency in the presence of altruistic autonomous vehicles.\footnote{Due to space constraints, we defer all proofs to the supplementary materials.}
\end{itemize}

Our development considers how to find optimal equilibria. This does not directly inform a planner on how to achieve such equilibria, but is the first step towards doing so. 

\noindent\textbf{Related work.} 
Our work is closely related to work in optimal routing, traffic equilibria, and Stackelberg strategies in routing games. In routing games, populations of users travel from source graph nodes to destination nodes via edges, and the cost incurred due to traveling on an edge is a function of the volume of traffic flow on that edge. Some works have found optimal routing strategies for these games~\cite{hearn1984convex, dafermos1972multiclass_user}. If all users choose their routes selfishly and minimize their own travel time this results in a \emph{Nash} or \emph{Wardrop Equilibrium} \cite{florian1995network, dafermos1980variationalinequalities}. Stackelberg routing games are games in which a leader controls some fraction of traffic, and the remainder of the traffic reaches a selfish equilibrium \cite{roughgarden2004stackelberg, bonifaci2010stackelberg, swamy2012effectiveness}. Many works bound the inefficiency due to selfish routing in these games including in the presence of traffic composed of multiple vehicle types \cite{roughgarden2002bad, correa2008geometric, perakis2007price, lazar2018poa}. However, these works on routing games consider a model of road latency in which latency increases as vehicle flow increases. This latency model does not capture congestion, in which jammed traffic has low flow \emph{and} high latency.

While not in the routing game setting, other works focus specifically on characterizing these congestion effects on traffic flow with multiple classes of users \cite{daganzo1997continuum,zhang2002kinematic, logghe2003heterogeneous, fan2015heterogeneous}. Krichene et al. combine these two fields by considering optimal Stackelberg strategies on networks of parallel roads with latency functions that model this notion of congestion~\cite{krichene2017stackelberg}. Their work considers a single traffic type (no autonomous vehicles) and the controlled traffic can be made to take routes with delay arbitrarily worse than the best available route. In contrast, we consider a routing game with two types of traffic and there is a distribution of the altruism level of the controlled autonomous traffic.

\smallspacingsection{Model and Problem Statement}
\noindent\textbf{Traffic Flow on Pure Roads.} We assume every road has a maximum flow $\zmax_i$. This occurs when traffic is in \emph{free-flow}-- when all vehicles travel at the nominal road speed. 

\emph{Traffic density vs traffic flow:} Adding more cars to a road that is already at maximum flow makes the traffic switch from free-flow to a congested regime, which decreases the vehicle flow. In the extreme case, at a certain density, cars are bumper-to-bumper and vehicle flow stops. Removing vehicles also decreases the flow since it is not going to speed up cars that are already traveling in free-flow at their nominal speed.
Fig.~\ref{fig:single_fund_diag_and_latency}(a)-- Fundamental Diagram of Traffic~\cite{DAGANZO1994269} -- illustrates this phenomenon, where flow increases linearly with respect to density with slope $v^f_i$ until it hits the critical density. The slope corresponds to the free-flow velocity on road $i$. At the critical point, flow decreases linearly until it is zero at the maximum density. This linear model matches our simulations (Section \ref{sct:simulations}).

\begin{figure}
	\centering
	\vspace*{-15px}
	\includegraphics[width=\linewidth]{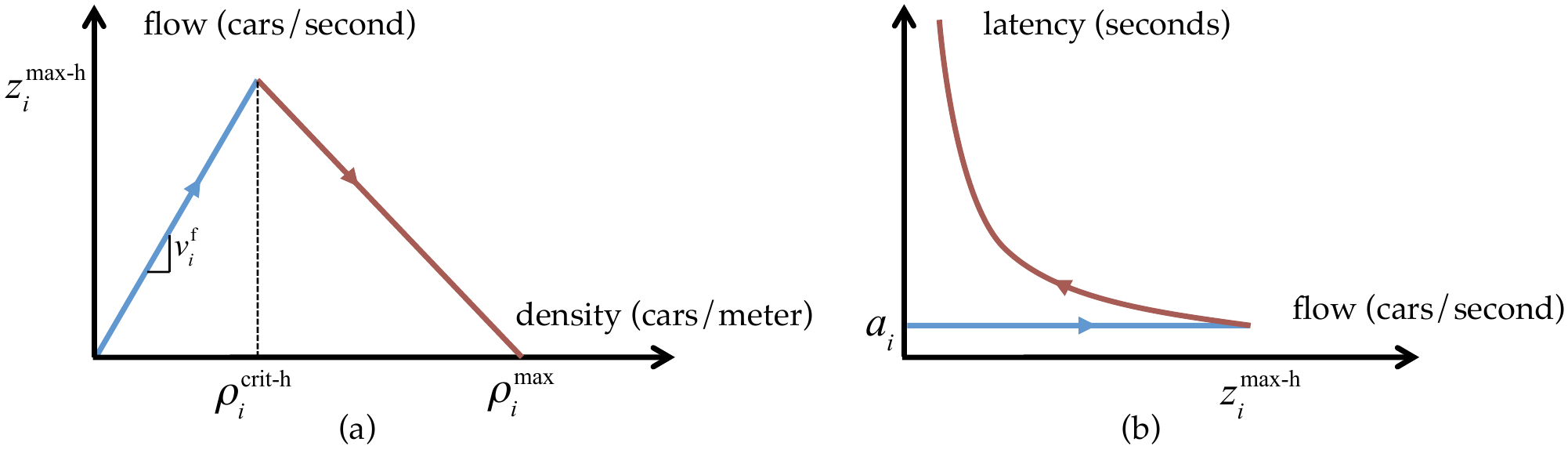}
	\vspace*{-18px}
	\caption{(a) Fundamental diagram of traffic for traffic of a pure road, relating vehicle density and flow. (b) Relationship between vehicle flow and latency. Blue and brown lines correspond to free-flow and congested regimes; arrowheads show increasing density.}
	\vspace*{-18px}
	\label{fig:single_fund_diag_and_latency}
\end{figure}

\emph{Traffic flow vs road latency: }The relationship between vehicle flow and road latency reflects the same free-flow/congested divide. In free-flow the road has constant latency, as all vehicles are traveling at the nominal velocity. In the congested regime, however, vehicle flow decreases as latency increases -- in congestion, a high density of vehicles is required to achieve a low traffic flow, resulting in high latency. This is represented in Fig.~\ref{fig:single_fund_diag_and_latency}~(b), where in free-flow, latency is constant at free-flow latency $a_i$ for any amount of flow up to the maximum flow on a road. The brown curve above corresponds to the congested regime, in which latency \emph{increases} as vehicle flow decreases.\footnote{Though congested flow may be unstable, it can be stabilized with a small number of autonomous vehicles \cite{wu2018stabilizing}, adding an additional constraint to later optimizations.}

\noindent\textbf{Traffic Flow on Mixed-Autonomy Roads.} We assume that on mixed-\linebreak autonomy roads, the autonomous vehicles can coordinate with one another and potentially form platoons to help with the efficiency of the road network.
We now extend the traffic model on pure roads to mixed-autonomy settings as shown in Fig.~\ref{fig:mixed_fund_diag_and_laten}~(a).
Assuming that neither the nominal velocity nor the maximum density changes, the critical density at which traffic becomes congested will now shift and increase with \emph{autonomy level} of the road. This can be explained as autonomous vehicles do not require a headway as large as that needed by human drivers due to platooning benefits.

\begin{figure}
	\centering
	\vspace*{-15px}
	\includegraphics[width=\linewidth]{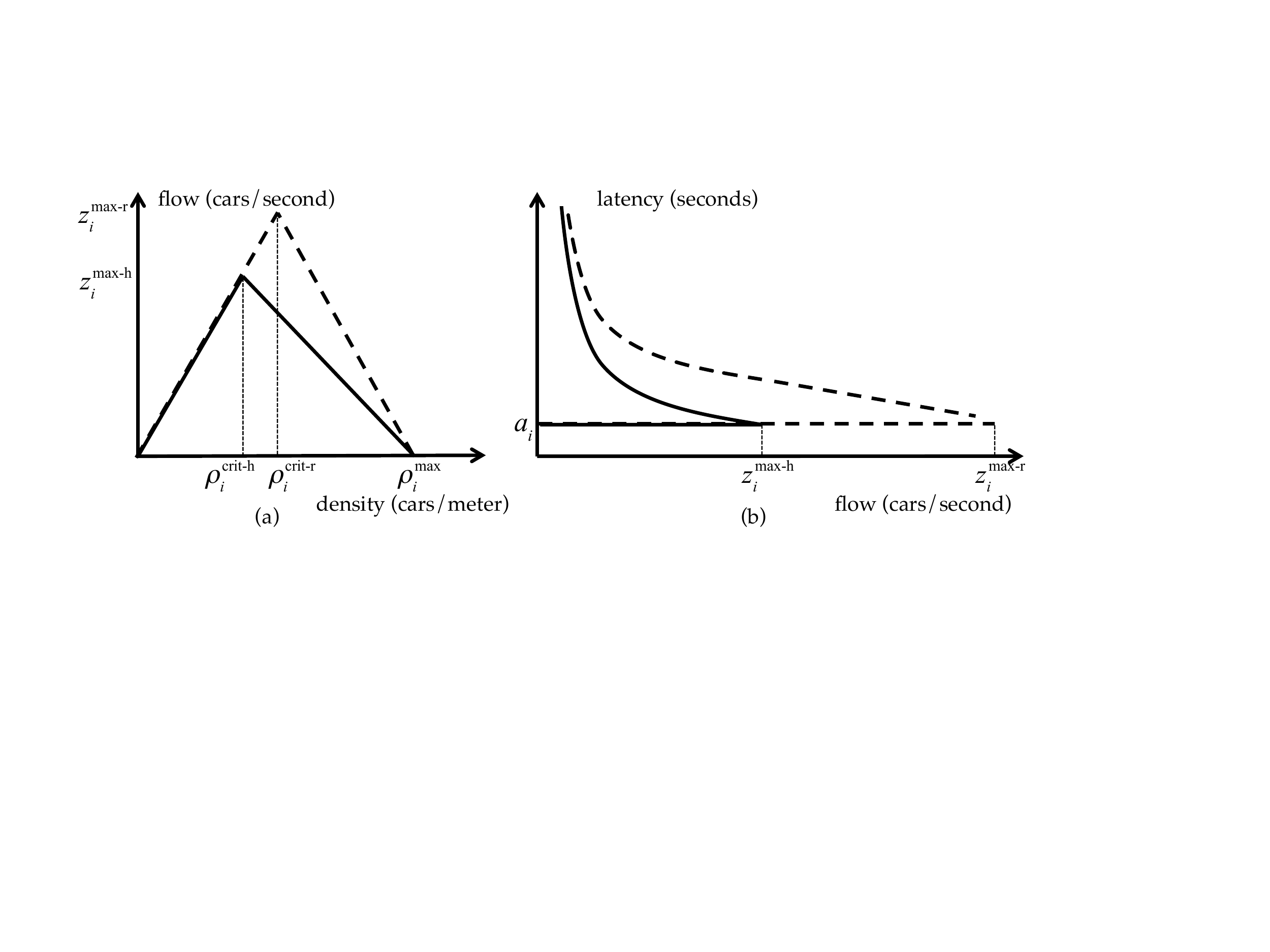}
	\vspace*{-18px}
	\caption{\textbf{(a)} The fundamental diagram of traffic for roads with all human-driven (solid) and all autonomous (dashed) vehicles. Congestion begins at a higher vehicle density as autonomous vehicles require a shorter headway when following other vehicles. Superscripts $\text{-h}$ and $\text{-r}$ denote parameters corresponding to purely human and purely autonomous traffic, respectively. \textbf{(b)} The relationship between vehicle flow and latency also changes in the presence of autonomous vehicles. Free-flow speed remains the same but maximum flow on a road increases.}
	\vspace*{-16px}
	\label{fig:mixed_fund_diag_and_laten}
\end{figure}

To formalize the relationship between autonomy level and critical density on road $i$, we assume at nominal velocity autonomous vehicles and humans require headway $\bar{h}_i$ and $h_i$, respectively, with $\bar{h}_i\!\le\!h_i$. This inequality reflects the assumption that autonomous vehicles can maintain a short headway, regardless of the type of vehicle they are following. Then, if $x_i$ and $y_i$ denote the flow of human-driven and autonomous vehicles respectively, the critical density is
\begin{align*}
	\rhocrit_i(x_i,y_i)= \frac{1}{\frac{y_i}{x_i+y_i}\bar{h}_i + \frac{x_i}{x_i+y_i}h_i + L} \;,
\end{align*}
where all vehicles have length $L$.
Here the denominator represents the average length from one car's rear bumper to the preceding car's rear bumper when all cars follow the vehicle in front of them with nominal headway. Note that critical density is expressed here as a function of flows $x_i$ and $y_i$. This quantity can also be expressed as a function of \emph{autonomy level} $\alpha_i(x_i,y_i)=\frac{y_i}{x_i+y_i}$, i.e., the fraction of vehicles on road $i$ that are autonomous. Since flow increases linearly with density until hitting $\rhocrit_i$, the maximum flow can also be expressed as a function of autonomy level: $\zmax_i(x_i,y_i)=v^f_i \rhocrit_i(x_i,y_i).$

The latency on a road is a function of vehicle flows on the road as well as a binary argument $s_i$, which indicates whether the road is congested:
\vspace{-6pt}
\begin{align}\label{eq:latency}
	\ell_i(x_i,y_i,s_i)=
	\begin{cases}
	\frac{d_i}{v^f_i} \; , & s_i=0, \\
	d_i(\frac{\rhomax_i}{x_i + y_i}+ \frac{\rhocrit_i(x_i,y_i)-\rhomax_i}{\zmax_i(x_i,y_i)})\; , & s_i=1.
	\end{cases} 
\end{align}
\vspace{-6pt}

We define the free-flow latency on road $i$ as $a_i:=d_i/v^f_i$, where $d_i$ is the length of road $i$. Fig.~\ref{fig:mixed_fund_diag_and_laten}~(b) illustrates the effect of mixed autonomy on latency. 

\begin{wrapfigure}{r}{0.36\textwidth}
	\centering
	\vspace{-22pt}
	\includegraphics[width=\linewidth]{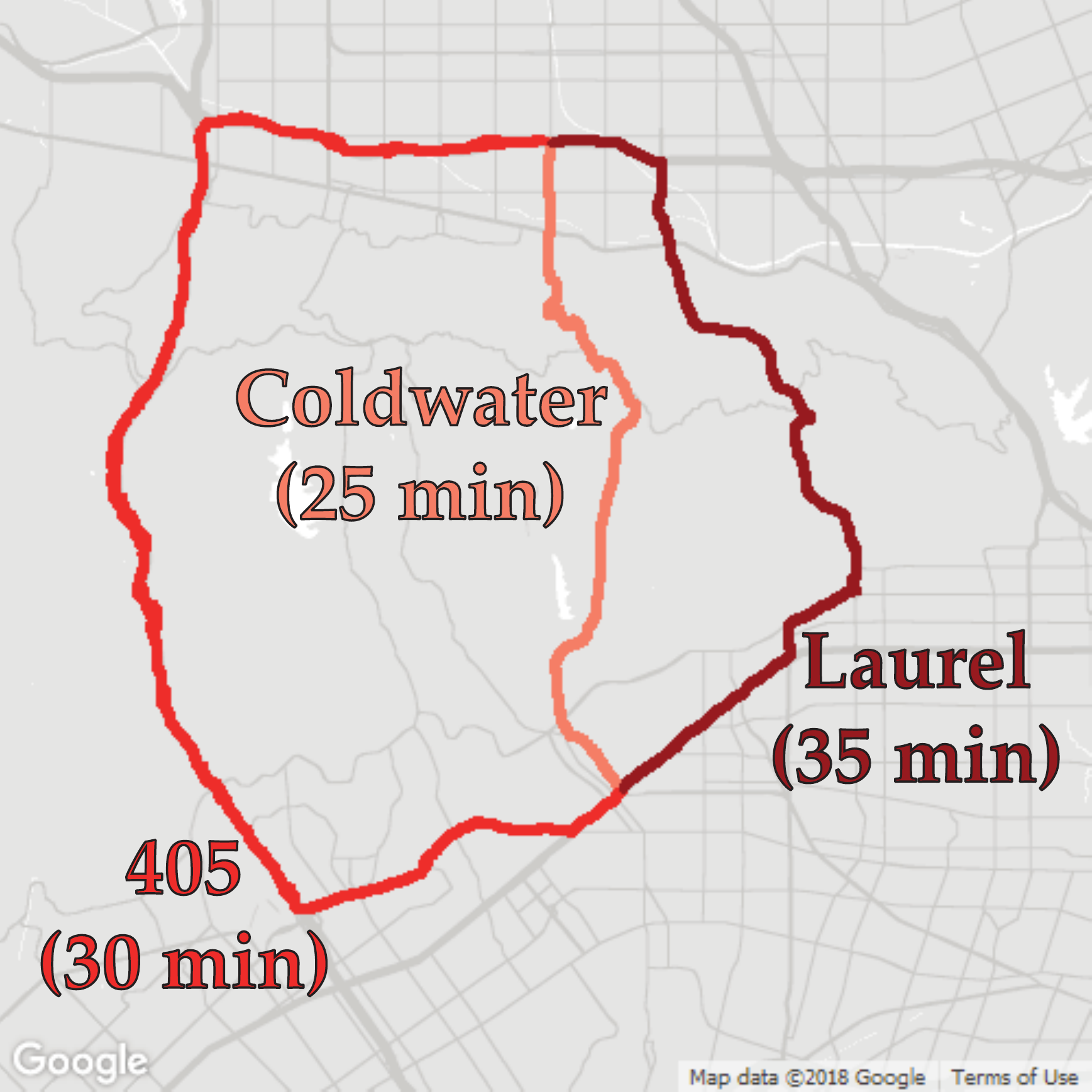}
	\vspace{-20pt}
	\caption{The map with three possible routes for the example.}
	\vspace{-10pt}
\end{wrapfigure}
\noindent\textbf{Demonstrative Example.} 
We use the following running example to intuitively describe the concepts formalized in Main Results (Sec.~\ref{sct:main_results}).
Imagine driving on a Friday afternoon in Los Angeles, where you plan to drive from the Beverly Hills library to the Valley. The most direct route, Coldwater Canyon, takes 25 minutes without traffic. Taking the 405 freeway would be 30 minutes, and Laurel Canyon would be 35 minutes in free-flow. Unfortunately, it is Friday afternoon and the flow is anything but free!
Let us assume that everyone is in the same predicament -- meaning that all traffic on these roads is from people with the same start and end points. People would only take Laurel Canyon if Coldwater and the 405 were congested to the point that each take at least 35 minutes. Further, any route that people use will have the same latency, otherwise they would switch to the quicker route. Fig.~\ref{fig:example_equilibria} illustrates three such equilibria with varying delays.

\begin{figure}
	\centering
	\vspace*{-18px}
	\includegraphics[width=1\linewidth]{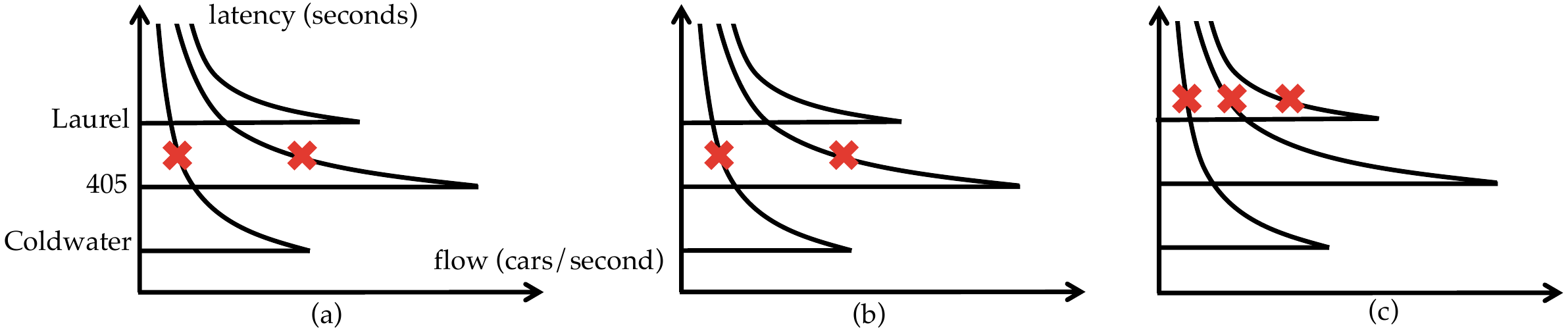}
	\vspace*{-18px}
	\caption{Illustration of equilibria in the network. Equilibria can have one road in free-flow and others congested (a) or all used roads can be congested (b, c).}
	\vspace*{-18px}
	\label{fig:example_equilibria}
\end{figure}

Krichene et al. have shown that for any given volume of traffic composed of only human drivers, there exists one \emph{best equilibrium} which is an equilibrium where one road is in free-flow and all other used roads are congested -- and the number of roads used is minimized \cite{krichene2017stackelberg}. \emph{However, when there are multiple types of vehicles that affect congestion differently, as we see in mixed-autonomy, there can be multiple equilibria with the same aggregate delay.}

Concretely, autonomous vehicles can form \emph{platoons} to increase the free-flow capacity of a road, an effect that is accentuated on freeways. This effect can be used not only to find equilibria that minimize aggregate delay, but also to find equilibria, within the set of delay-minimizing equilibria, that can accommodate extra unforeseen flow demand. As an example, Fig.~\ref{fig:example_mixed_aut} (a) and (b) have the same delay, whether all autonomous flow (green) is the 405 or is split between the 405 and Coldwater. However, (b) has higher autonomy level on the 405 than (a) and can therefore accommodate more additional flow on the 405, since it matches its autonomous vehicles with the roads that benefit most from them. If a social planner can dictate equilibrium routing but does not have perfect information about flow demands, using a routing such as in (b) can make the routing more \emph{robust} to unforeseen demand.

\begin{figure}
	\centering
	\vspace*{-14px}
	\includegraphics[width=1\linewidth]{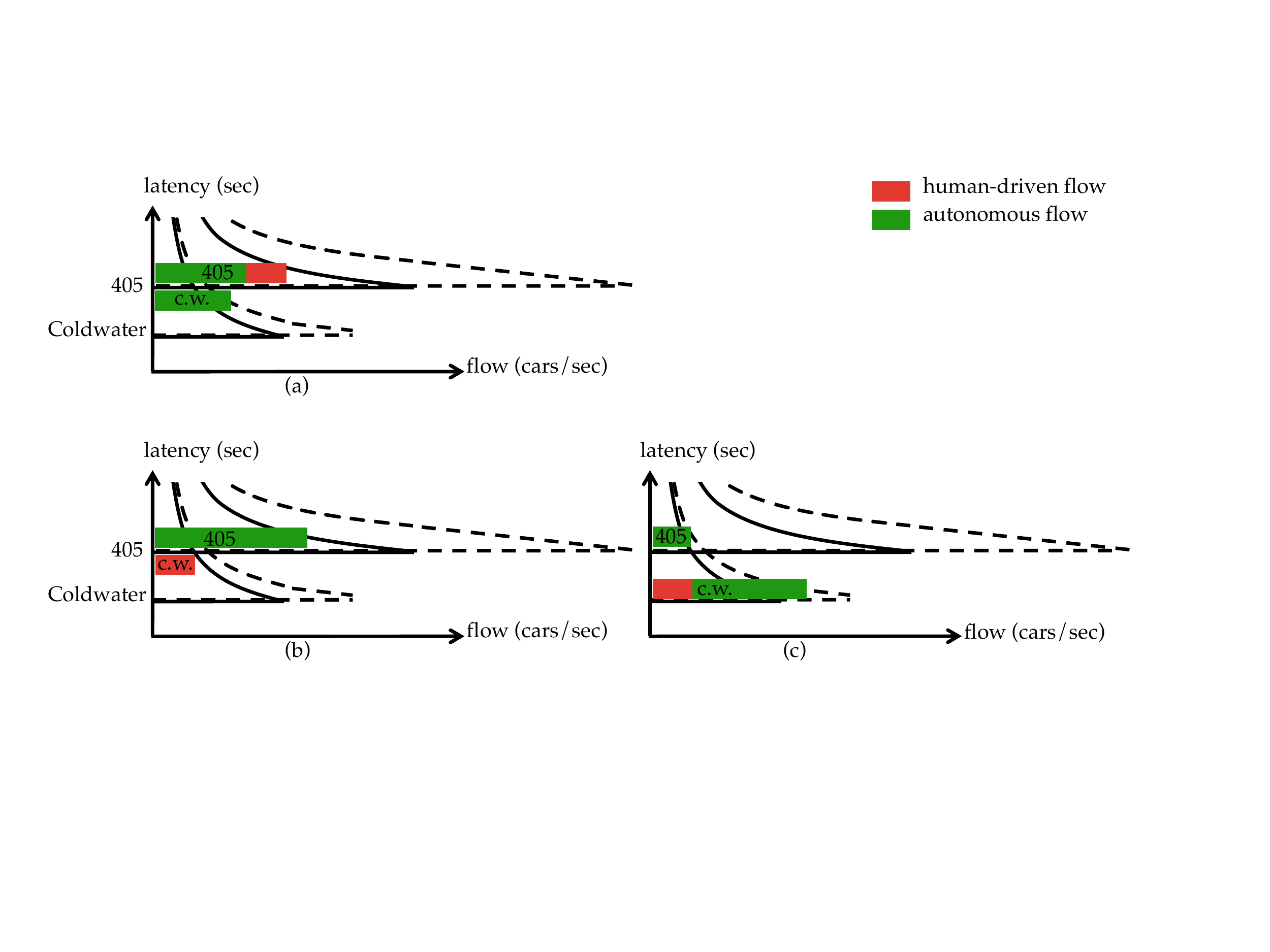}
	\vspace*{-18px}
	\caption{(a) and (b) represent two selfish equilibria with the same social cost. The equilibrium with all autonomous flow on the 405 freeway (b) is more robust to additional unforeseen flow demand. Solid (resp. dashed) lines characterize the road with only human drivers (resp. autonomous vehicles). The 405 has maximum flow that benefits more from autonomy than the canyons -- by routing all autonomous traffic onto the 405, a social planner can make the routing more robust to unforeseen demand. If some autonomous users are altruistic, a social planner can send them on the 405 while Coldwater is uncongested (c), leading to a lower overall travel time.}
	\vspace*{-18px}
	\label{fig:example_mixed_aut}
\end{figure}

The social cost can further be improved if users of autonomous vehicles are \emph{altruistic} and are willing to endure a delay longer than the quickest route available in order to improve traffic. Assume the autonomous cars are owned by entities concerned for the public good and consider a population of autonomous car users who, in exchange for cheaper rides, are willing to accept a delay of 20\% longer than the quickest available route. In this case, a social planner can send some autonomous users on the 405 so that Coldwater is in an uncongested state, yielding a social cost lower than that of the best-case selfish equilibria (Fig.~\ref{fig:example_mixed_aut}~(c)).

Before presenting our main results, we also define our network and objective.
\noindent\textbf{Network Objectives.} We consider a network of $N$ parallel roads, with indices $\{1,\ldots,n\}$. For convenience we assume no two roads have identical free-flow latencies, and accordingly, order our roads in order of increasing free-flow latency. We use $\midcs=\{1,\ldots,m\}$ to denote the set of roads with indices $1$ through $m$.

We consider an inelastic demand of $\barx$ and $\bary$ (human-driven and autonomous vehicle flow, respectively), meaning that the total flow of vehicles wishing to travel across the network is independent of the delays on the roads. We describe the network state by $(\bx, \by, \bs)$, where $\bx$, $\by \in \mathbb{R}^N_{\ge 0}$ and $\bs \in \{0,1\}^N$. A \emph{feasible routing} is one for which $x_i + y_i \le \zmax_i(x_i,y_i)$ for all roads, where $\zmax_i$ denotes the maximum flow on road $i$. We denote total flow on a road by $z_i:=x_i+y_i$. We are interested in finding a routing, i.e. allocation of vehicles into the roads, that minimizes the total latency experienced by all vehicles, $C(\bx,\by,\bs)=\sum_{i\in \nidcs}(x_i + y_i)\ell_i(x_i,y_i,s_i)$, while satisfying the demand, i.e. $\sum_{i\in \nidcs}x_i = \barx$ and $\sum_{i\in \nidcs}y_i = \bary$.
Further, we constrain this optimization based on \emph{selfishness} or \emph{altruism} of the vehicles, which we will define in Section~\ref{sct:main_results}.

\smallspacingsection{Main Results}\label{sct:main_results}
We now make precise the aforementioned notions of selfishness, robustness, and altruism. We develop properties of the resulting equilibria, and using those, provide polynomial-time algorithms for computing optimal vehicle flows.

\noindent\textbf{Selfishness.} Human drivers are often perceived as selfish, i.e., they will not take a route with long delay if a quicker route is available to them. If all drivers are selfish this leads to a \emph{Nash}, or \emph{Wardrop Equilibrium}, in which no driver can achieve a lower travel time by unilaterally switching routes \cite{florian1995network, dafermos1980variationalinequalities}. In the case of parallel roads, this means that all selfish users experience the same travel time. \vspace{-4pt}
\begin{definition}
	The \textbf{\emph{longest equilibrium road}} is the road with maximum free-flow latency that has delay equal to the delay experienced by selfish users. Let $\meq$ denote the index of this road. We then use $\text{NE}(\barx,\bary,\meq)$ to denote the set of Nash Equilibria with \emph{longest equilibrium road} having index $\meq$. 
\end{definition} 
\vspace{-10pt}
\begin{definition}
	The \textbf{\emph{longest used road}} is the road with maximum free-flow latency that has positive vehicle flow of any type on it. We use $\mall$ to denote the index of this road; if all vehicles in a network are selfish then $\meq=\mall$.
\end{definition}
\vspace{-5pt}

We define the set of \textbf{\emph{Best-case Nash Equilibria (BNE)}} as the set of feasible routings in equilibrium that minimize the total latency for flow demand $(\barx, \bary)$, denoted  $\text{BNE}(\bar{x}, \bar{y})$. The following theorem provides properties of the set of BNE for mixed-autonomy roads (for pure roads see \cite{krichene2017stackelberg}).
\vspace{-4pt}
\begin{theorem}\label{thm:BNE}
	There exists a road index $\mstareq$ such that all routings in the set of BNE have the below properties. Further, this index $\mstareq$ is the minimum index such that a feasible routing can satisfy the properties:
	\begin{enumerate}[nosep]
		\item road $\mstareq$ is in free-flow,
		\item roads with index less than $\mstareq$ are congested with latency $a_{\mstareq}$, and
		\item all roads with index greater than $\mstareq$ have zero flow.
	\end{enumerate}
\end{theorem}
\vspace{-4pt}

The running example in Fig.~\ref{fig:example_mixed_aut} shows that there does not necessarily exist a unique BNE. However, as estimating flow demand can be difficult, we prefer to choose a BNE that can best incorporate additional unforeseen demand. We thus develop the notion of \emph{robustness}.
\vspace{-4pt}
\begin{definition}\label{def:robustness}
	The \emph{robustness} of a routing in the set of Nash Equilibria is how much additional traffic (as a multiple of the original demand flow), at the overall autonomy level, can be routed onto the free-flow road. Formally, if $(\bx, \by)\in \text{NE}(\barx,\bary, m)$, then it has robustness:
	\vspace{-4pt}
	\begin{align}
		& \beta(\bx, \by,\barx,\bary,s_{m}):= \label{eq:robustness} \\
		& \quad \begin{cases}
		\max \; \gamma \; \text{s.t.} \; x_{m} + y_{m} + \gamma(\barx + \bary) \le \zmax_{m}(x_{m} + \gamma \barx, y_{m} + \gamma \bary) & s_{m}=0 \\
		0 & s_{m} = 1 \; .
		\end{cases} \nonumber
	\end{align}
\end{definition}

We use \textbf{\emph{Robust Best Nash Equilibria (RBNE)}} to refer to the subset of BNE that maximize robustness. The RBNE may not be unique.

\noindent\textbf{Altruism.} It is possible that some passengers, especially those in autonomous vehicles, can be incentivized to use routes that are not quickest for that individual, but instead lead to a lower social cost (Fig.~\ref{fig:example_mixed_aut}). We use the term \emph{altruism profile} to refer to the distribution of the degree to which autonomous vehicles are willing to endure longer routes. For computational reasons, we consider altruism profiles with a finite number of altruism levels.

\begin{figure}
	\centering
	\vspace*{-15px}
	\includegraphics[width=1\linewidth]{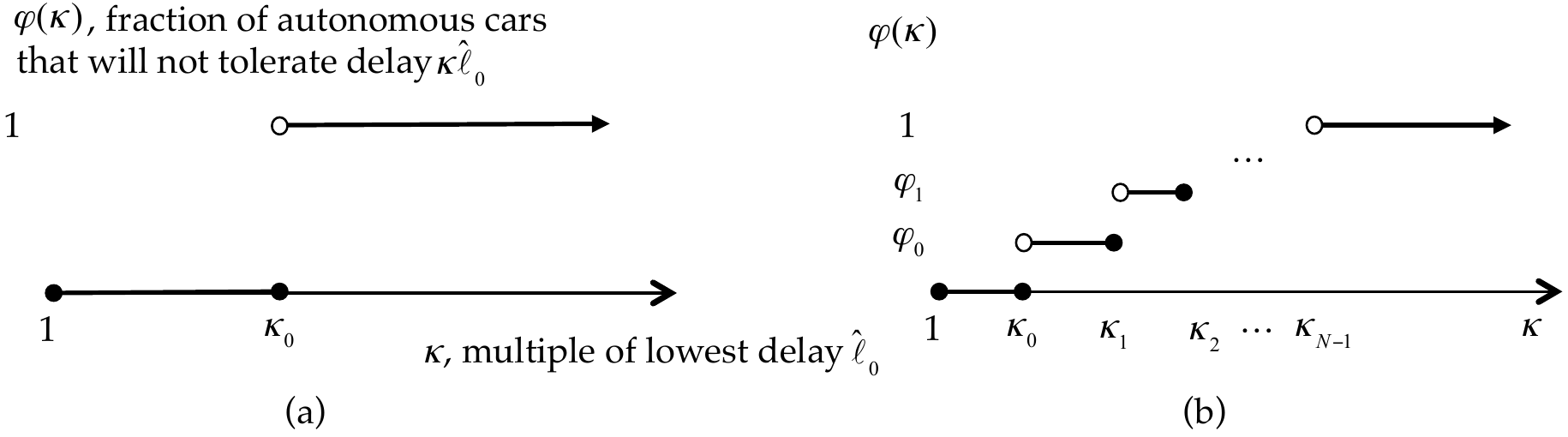}
	\vspace*{-15px}
	\caption{Altruism profiles. A fraction $\altruismFn(\altruismArg)$ of autonomous users will not accept delay greater than $\altruismArg$ times that of the quickest available route. \textbf{(a)} Users will tolerate delay of up to $\altruismUnif$ times that of the quickest route. \textbf{(b)} Users have multiple altruism levels.}
	\vspace*{-18px}
	\label{fig:altruism}
\end{figure}

Formally, we define $\altruismFn: \mathbb{R}_{\ge 0} \rightarrow {[}0,1{]}$ to represent the altruism profile as a nondecreasing function of a delay value that is mapped to ${[}0,1{]}$. A volume of $\altruismFn(\altruismArg)\bary$ autonomous flow will reject a route incurring delay $\altruismArg$ times the minimum route delay available, which we denote by $\eqDelay$. If autonomous users have a uniform altruism level as in Fig.~\ref{fig:altruism}~(a), we call them \emph{$\altruismUnif$-altruistic} users, where $\altruismUnif$ is the maximum multiple of the minimum delay that autonomous users will accept. Users may have differing altruism levels, as in Fig.~\ref{fig:altruism}~(b). We use $\altruismLevSet$ to denote the set of altruism levels, with cardinality $|\altruismLevSet|$. Accordingly, a feasible routing $(\bx, \by, \bs)$ is in the set of \textbf{\emph{Altruistic Nash Equilibria (ANE)}} if 
\begin{enumerate}[nosep]
	\item all routes with human traffic have latency $\eqDelay \le \ell_i(x_i,y_i,s_i)$ $\forall i \in \nidcs$ and
	\item for any $\ell \geq 0$, a volume of at least $\altruismFn(\ell/\eqDelay) \; \bar{y}$ autonomous traffic experiences a delay less than or equal to $\ell$. Note that it is enough to check this condition for $\ell = \ell_i(x_i,y_i,s_i)$ for all $i$.\footnote{Note that this is similar to the notion of \emph{epsilon-approximate Nash Equilibrium} \cite{nisan2007algorithmic}, but with populations playing strategies that bring them within some factor of the best strategy available to them, with each population having a different factor.}
\end{enumerate}

We denote the set of routings at Altruistic Nash Equilibria with demand $(\barx,\bary)$, equilibrium latency $\eqDelay$, and altruism profile $\altruismFn$ as $ANE(\barx,\bary,\eqDelay,\altruismFn)$. The set of \textbf{\emph{Best-case Altruistic Nash Equilibria (BANE)}} is the subset of ANE with routings that minimize total latency.
\vspace{-4pt}
\begin{theorem}\label{thm:BANE}
	There exist a longest equilibrium road $\mstareq$ and a longest used road $\mstarall$ with $\mstareq \le \mstarall$, such that all routings in BANE have the following properties:
	\begin{enumerate}[nosep]
		\item roads with index less than $\mstareq$ are congested,
		\item roads with index greater than $\mstareq$ are in free-flow,
		\item roads with index greater than $\mstareq$ and less than $\mstarall$ have maximum flow.
	\end{enumerate}
\end{theorem}
\vspace{-10pt}
\begin{remark}
	One may be tempted to think that, like in BNE, road $\mstareq$ will be in free-flow. Further, one might think that $\mstareq$ is the minimum index such that all selfish traffic can be feasibly routed at Nash Equilibrium. We provide counterexamples to these conjectures in the supplementary materials.
\end{remark}
\vspace{-6pt}

\noindent\textbf{Finding the Best-case Nash Equilibria.} In general, the Nash Equilibrium constraint is a difficult combinatorial constraint. Theorem \ref{thm:BNE} however states that we can characterize the congestion profile of the roads by finding the minimum free-flow road such that Nash Equilibrium can be feasibly achieved. Once this is found, we select a feasible routing that maximizes robustness. This is formalized as follows: find the minimum $\meq$ such that $\text{NE}(\barx,\bary, \meq)$ is nonempty. Then the RBNE is the set of routings that maximize the robustness of road $\meq$:
\vspace{-6pt}
\begin{align}\label{opt:robust_BNE}
& \qquad \mstareq = \argmin_{\meq \in \nidcs} a_\meq \; \text{s.t.} \; \text{NE}(\barx,\bary, \meq) \neq \emptyset \; \text{, then} \nonumber \\
&\text{BNE}(\barx,\bary, \mstareq) = \argmax_{(\bx,\by,\bs) \in \text{NE}(\barx,\bary, \mstareq) } \beta(\bx,\by,\barx, \bary, s_{\mstareq}) \; .
\end{align}
\vspace{-14pt}
\begin{theorem}\label{thm:compute_RBNE}
	Finding a solution to \eqref{opt:robust_BNE} is equivalent to finding a routing in the set of RBNE, if any exist. Further, \eqref{opt:robust_BNE} can be solved in $O(N^4)$ time.
\end{theorem}
\vspace{-4pt}

We provide a sketch of the algorithm for finding an element of the RBNE:
\begin{enumerate}[nosep]
	\item Search over $\meq$, the longest equilibrium road, to find the minimum $\meq$ that is feasible for required flow demands at Nash Equilibrium.
	\item For this $\meq$, find the routing that maximizes the robustness of road $\meq$. This can be formulated as a linear program, which has cubic computational complexity in the number of roads \cite{gonzaga1992path}.
\end{enumerate}

We provide the linear program to find feasibility and routing, which maximizes robustness while at free-flow equilibrium, subject to demand constraints.
\begin{align}
	&\max_{\bx, \by \in \mathbb{R}^N_{\ge 0}, \gamma \ge 0} \gamma \; \text{s.t.} \sum_{i \in \mstareqidcs}x_i = \barx, \, \sum_{i \in \mstareqidcs}y_i = \bary, \; \ell_i(x_i,y_i,1) = a_\mstareq \; \forall i \in {[}m-1{]}, \nonumber \\
	&\qquad \; \, x_\mstareq + y_\mstareq + \gamma(\barx + \bary) \le \zmax_\mstareq(x_\mstareq + \gamma \barx, y_\mstareq + \gamma \bary) \; . \nonumber
\end{align}

\noindent\textbf{Finding the Best-case Altruistic Nash Equilibria.} To find an element of the BANE, we need to solve:
\vspace{-6pt}
\begin{align}\label{opt:BANE}
\argmin_{\substack{\meq \in \nidcs, \; \; \eqDelay \in {[}a_\meq,a_{\meq+1}{)}, \; \; (\bx,\by,\bs) \in \text{ANE}(\barx,\bary,\eqDelay,\altruismFn)} } C(\bx,\by) \; .
\end{align}
As demonstrated in the supplementary material, the BANE no longer has monotonicity in $\meq$, the index of the longest equilibrium road. Further, road $\meq$ may not be in free-flow in the set of BANE. We do however have a degree of monotonicity in the latency on road $\meq$: for a fixed $\meq$, optimal social cost is minimized at the minimum feasible latency on road $\meq$.
\vspace{-4pt}
\begin{theorem}\label{thm:compute_BANE}
	Finding a solution to \eqref{opt:BANE} is equivalent to finding a routing in the set of BANE, if any exist. Further, \eqref{opt:BANE} can be solved in $O(|\altruismLevSet|N^5)$ time, where $|\altruismLevSet|$ is the number of altruism levels of autonomous vehicle users.
\end{theorem}
\vspace{-4pt}

In the full proof we show that for any given $\meq$, there are a maximum of $|\altruismLevSet|N$ critical points to check for the equilibrium latency, $\eqDelay$. We sketch the algorithm for finding a routing in the set of BANE:
\begin{enumerate}[nosep]
	\item Enumerate through all options for $\meq$.
	\item For each $\meq$, enumerate through the $|\altruismLevSet|N$ critical points of $\eqDelay$.
	\item For each of these combinations, find the routing that maximizes the autonomous flow on roads $\meqidcs$, while incorporating all human flow on these roads. This can be formulated as a linear program. Then find optimal autonomous routing on the remaining roads (order $N$ calculations).
\end{enumerate}

\noindent\textbf{Improvement can be unbounded.} We motivate the schemes described above by demonstrating that without them, aggregate latency can be arbitrarily worse than with them (Fig.~\ref{fig:BANE_BNE_unbound_improve_case_study}). First we show the cost at BNE can be arbitrarily worse than that at BANE, even when autonomous users have arbitrarily low altruism.

\begin{figure}
	\centering
	\vspace*{-18px}
	\includegraphics[width=\linewidth]{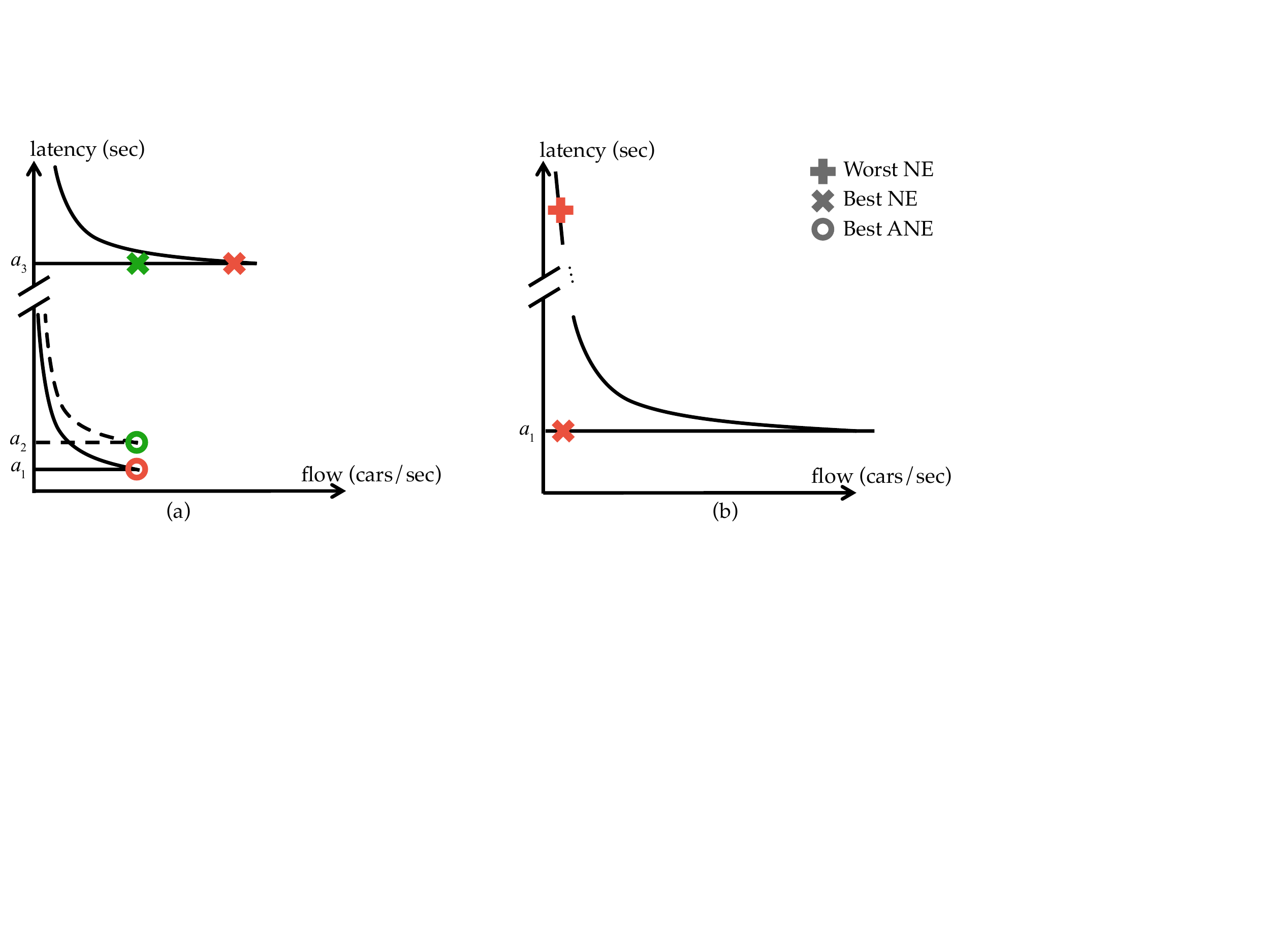}
	\vspace*{-18px}
	\caption{Illustration showing unbounded improvement in (a) altruistic NE vs. best selfish NE, and (b) best selfish NE vs. worst NE. Red and green represent regular and autonomous vehicle flow, +'s are worst-case Nash Equilibrium routing, x's are best-case NE routing, and circles are best-case altruistic NE routings.}
	\vspace*{-18px}
	\label{fig:BANE_BNE_unbound_improve_case_study}
\end{figure}

Consider three roads with free-flow latencies $a_3\!>\!a_2\!>\!a_1$, and human-driven and autonomous flow demands $\barx$ and $\bary$ relative to maximum flows $\zmax_1(\barx,0)=\barx$, $\zmax_2(0,\bary)=\bary$, and $\zmax_3(\barx,\bary) \ge \barx + \bary$. This means human-driven vehicles can fit on road 1 with autonomous vehicles on road 2, or both vehicle types can fit on roads 1, 2, and 3 if all are selfish, as all vehicles cannot fit on roads 1 and 2 at the same latency. Let autonomous vehicles have uniform altruism level $\altruismUnif\!=\! \frac{a_2}{a_1}$, which can be arbitrarily close to $1$. In BANE, autonomous vehicles will be on road 2 and human-driven vehicles on road 1. However in BNE, all flow will be on road 3, as the others are fully congested. Then, $\frac{C^\text{BNE}}{C^\text{BANE}} = \frac{a_3 (\barx + \bary)}{a_1\barx + a_2 \bary} \geq \frac{a_3}{a_2}$, which can be arbitrarily large.

Now consider one road with very large critical density and a small demand $\barx$. This road can exist at free-flow latency at low density or in a highly congested state at high density and serve flow $\barx$ either way. Then, $\lim_{\barx \rightarrow 0^+}\frac{\ell_i(\barx,0,1)}{\ell_i(\barx,0,0)} \rightarrow \infty$, which shows a NE can be arbitrarily worse than a BNE.

\smallspacingsection{Experiments}\label{sct:simulations}
\noindent\textbf{Numerical example of gained utility by BNE and BANE.} 
\begin{wrapfigure}{r}{0.37\textwidth}
	\centering
	\vspace*{-20px}
	\includegraphics[width=\linewidth]{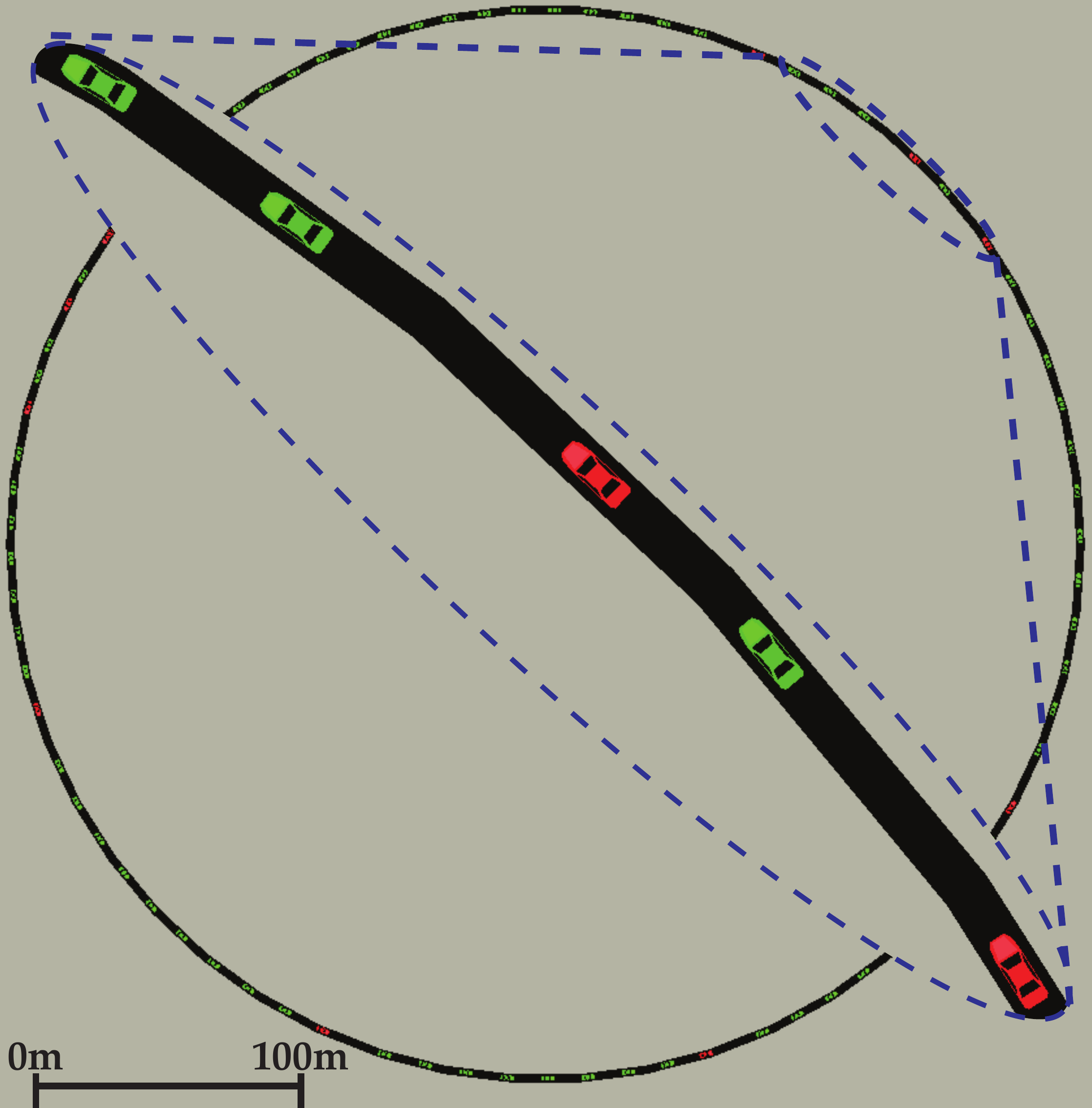}
	\vspace{-16px}
	\caption{The simulation visualization for Road 1. Regular vehicles (red) have larger headway than autonomous ones (green).
	}
	\vspace*{-22pt}
	\label{fig:circular_road}
\end{wrapfigure}
Next, in order to numerically show the utility gained by RBNE (or BNE) against another NE and by altruistic autonomy in a more realistic setting, we create a scenario with two parallel residential roads of length $400\pi$ and $1000\pi$ meters. The speed limit is $13.9$ m/s ($\sim50$ kph). All vehicles are $5$ meters long and have headway of $\max(2,\tau v)$ meters with the vehicle in front, where $v$ represents the speed of the vehicle and $\tau$ is the reaction time ($2$ seconds for regular and $1$ second for autonomous vehicles). We want to allocate the flows of $0.3$ regular and $0.3$ autonomous vehicles per second.

We simulated this scenario using the traffic simulator SUMO \cite{krajzewicz2012recent}. To conveniently simulate the congestion effect, we designed circular and single-lane roads (Fig.~\ref{fig:circular_road}). We employed the original Krauss car following model for driver behavior \cite{krauss1997metastable}. We solved the optimizations using the CVX framework \cite{cvx, gb08}. For comparison, we also computed and simulated a congested NE.
\begin{wraptable}{r}{0.65\textwidth}
	\centering
	\vspace{-22px}
	\caption{Effect of Altruistic Autonomy}
	\vspace{-10px}
	\label{tab:simulation2}
	\begin{tabular}{@{}lccc@{}} 
		\toprule
		& \textbf{NE} & \textbf{RBNE} & \textbf{BANE }($\altruismUnif\!=\!2.5$)\\
		\midrule
		($x_1,y_1$) & \cellcolor{gray!50}($0.006, 0.252$) & \cellcolor{gray!50}($0.3,0.031$) & ($0.3,0.215$) \\ 
		($x_2,y_2$) & \cellcolor{gray!50}($0.294$,$0.048$) & ($0,0.269$) & ($0,0.085$) \\
		$C_T(\mathbf{x},\mathbf{y})$ & $324$ & $135.608$ & $61.5$ \\
		$C_S(\mathbf{x},\mathbf{y})$ & $297.052$ & $135.87$ & $68.64$ \\
		\bottomrule
	\end{tabular}
	\vspace*{-15px}
\end{wraptable}

In both calculations and simulations, \textbf{the BNE approximately halves the total cost compared to congested NE, and the BANE halves it again} (Table~\ref{tab:simulation2}). The altruism level is $\altruismUnif\!=\!2.5$. Gray cells represent congested roads. Subscripts $T$ and $S$ denote if the presented result is theoretical or simulated, respectively. Fig.~\ref{fig:simulation2} visualizes the solutions over flow-latency graph.
\begin{wrapfigure}{r}{0.35\textwidth}
	\centering
	\vspace*{-20px}
	\includegraphics[width=0.37\textwidth]{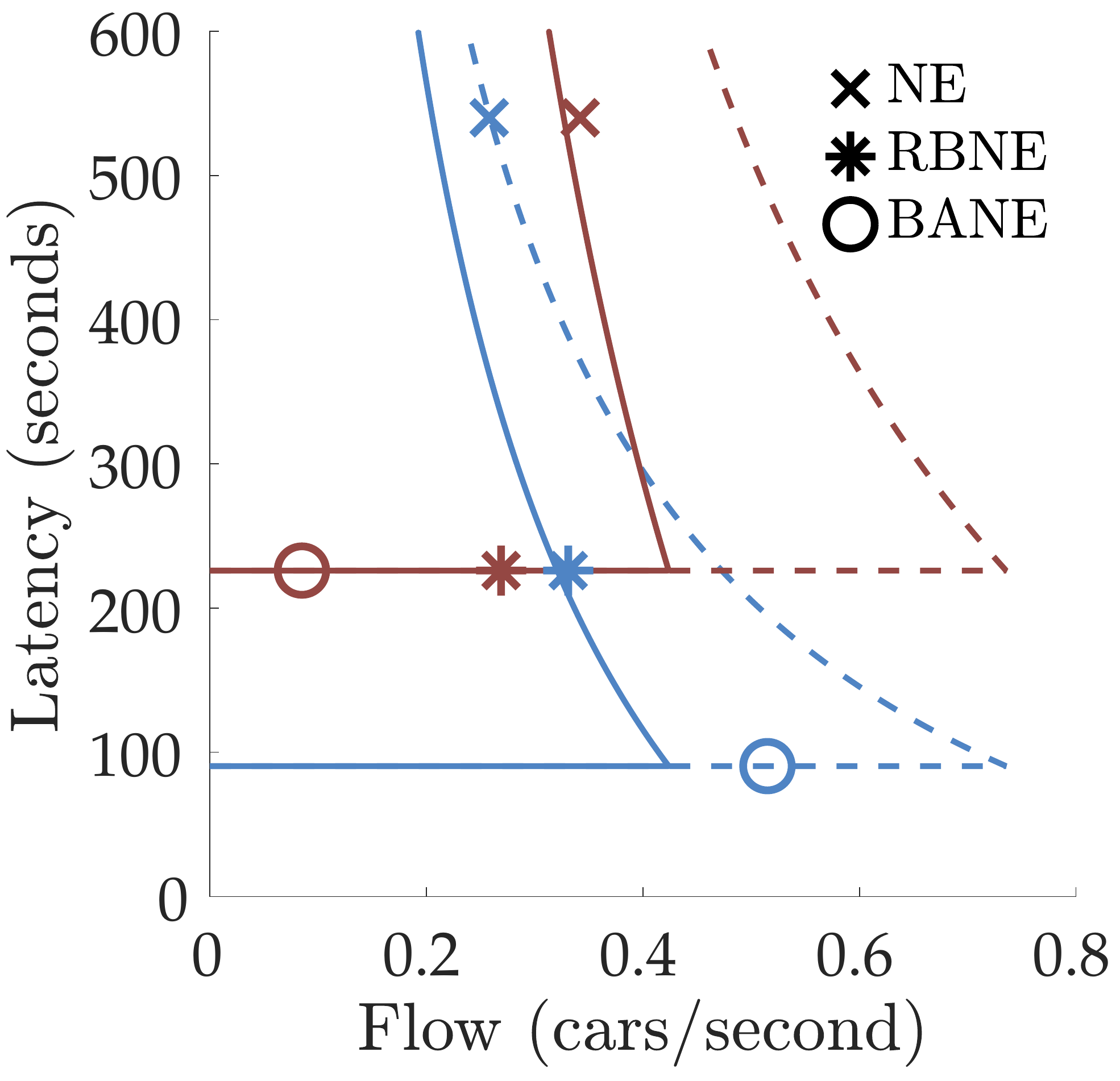}
	\vspace*{-18px}
	\caption{Three solutions to the scenario with two residential roads are visualized. Blue lines represent the shorter road, whereas brown is the longer one. Solid and dashed lines correspond to fully regular and fully autonomous configurations, respectively.
	}
	\vspace{-24px}
	\label{fig:simulation2}
\end{wrapfigure}

\noindent\textbf{Simulation of a 4-parallel-road network.} To show the properties of all the equilibria we defined, we simulated the following comprehensive scenario: We assume there are 4 parallel roads with different lengths from one point to another, two of which are residential roads where the speed limit is $13.9$ m/s, and the other two are undivided highways where the limit is $25.0$ m/s ($\sim\!90$ kph). The residential roads and highways have lengths $400\pi$, $600\pi$, and $800\pi$, $1000\pi$ meters, respectively. Sorting based on the free-flow latencies makes the order of lengths $(400\pi,\!800\pi,\!1000\pi,\!600\pi)$, and we use this order from this point on. We keep all other parameters the same as the previous simulation.

With this configuration, we first simulate Road 4 (residential road, $600\pi$ meters long) alone. The resulting fundamental diagram and the relationship between vehicle flow and latency are shown in Fig.~\ref{fig:simulation} along with the theoretical curves.
There exists a small mismatch between the theory and simulations, which is apparent only for $\zmax$. We conjecture that this is stemming from: \textbf{1) Shape:} We approximated circular roads using pentacontagons (regular polygons with $50$ edges), so the actual road lengths are slightly different. \textbf{2) Initiation:} We cannot initiate the vehicles homogeneously over the polygon. Instead, we initiate from the corners. This hinders our ability to allocate the maximum possible number of vehicles. \textbf{3) Discretization:} As the number of vehicles needs to be an integer, density values are heavily discretized. \textbf{4) Randomness:} Given the total number of vehicles for each road, our simulation initiates regular and autonomous vehicles drawn from a Bernoulli process with the given probability distribution.

To show the effect of changing $\altruismUnif$ in the altruistic case, we consider human flow demand of $0.4$ and autonomous flow demand of $1.2$ vehicles per second. For BANE, we simulated both $\altruismUnif\!=\!1.25$ and $\altruismUnif\!=\!1.5$, where we assume all autonomous vehicles are uniformly altruistic for simplicity.

\begin{figure}
	\centering
	\vspace*{-15px}
	\includegraphics[width=\linewidth]{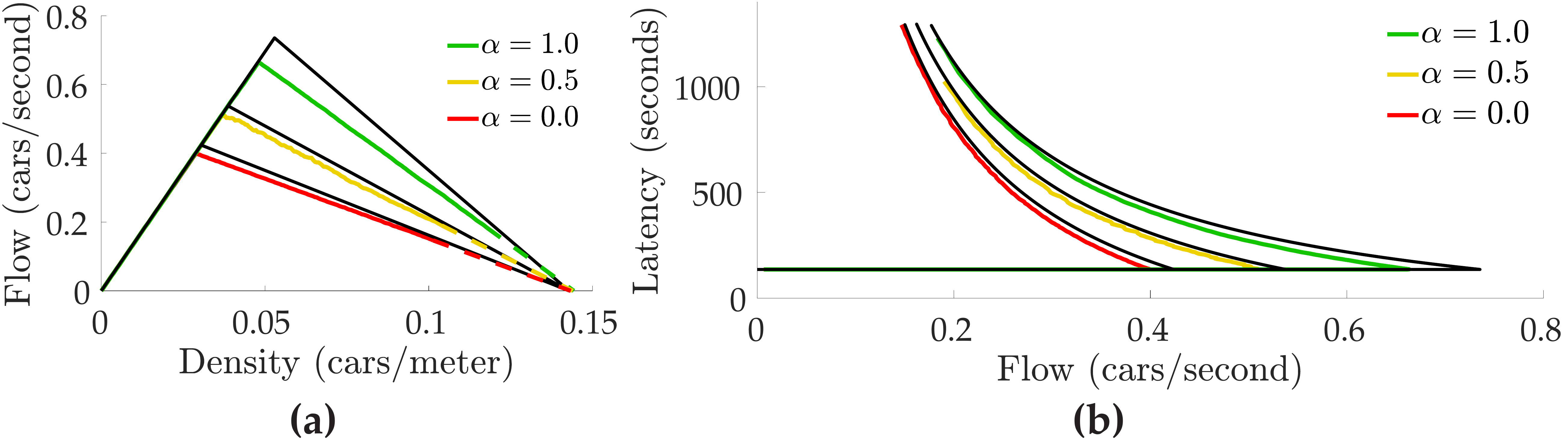}
	\vspace*{-19px}
	\caption{Simulation results of Road 4 for three different $\alpha$ values with theoretical values (black curves). The noisy appearance of the curves for $\alpha\!=\!0.5$ is due to randomness, while mismatches on $\rhocrit$ are due to the other imperfections. \textbf{(a)} The fundamental diagram of traffic. Due to simulation incapabilities regarding initiation, we cannot simulate
heavily congested regions. We used more lightly congested traffic measurements to project heavily congested cases (dashed lines). \textbf{(b)} The flow and latency relationship.}
	\vspace*{-21px}
	\label{fig:simulation}
\end{figure}

We then simulate the computed flow values using SUMO. We found that often the full calculated flow cannot be allocated in the simulation because of the mismatch between the theoretical maximum flow on a road and the simulated maximum flow (Fig.~\ref{fig:simulation}~(b)). In these cases, we kept $\alpha$-value for that road fixed and allocated flow with respect to Nash equilibrium. Hence, for BNE, RBNE and BANE, we added the remaining flow into road $\mstareq$, without breaking Nash equilibrium. For NE, we again kept $\alpha$ and the number of utilized roads fixed and swept latency values to satisfy flow conservation. These simulation results, as well as theoretically computed values for $\mathbf{x}$, $\mathbf{y}$, $\beta$, $C(\mathbf{x}, \mathbf{y})$ are given in Table~\ref{tab:simulation}.

\begin{table}
	\centering
	\vspace{-26px}
	\caption{Theoretical and Simulated Results}
	\label{tab:simulation}
	\begin{tabular}{@{}lccccc@{}} 
		\toprule
		& \textbf{NE} & \textbf{BNE} & \textbf{RBNE} & \textbf{BANE }($\altruismUnif\!=\!1.25$) & \textbf{BANE }($\altruismUnif\!=\!1.5$)\\
		\midrule
		($x_1,y_1$) & \cellcolor{gray!50}($0.036,0.277$) & \cellcolor{gray!50}($0.075,0.52$) & \cellcolor{gray!50}($0.391,0$) & \cellcolor{gray!50}($0.4,0.024$) & ($0.4,0.041$) \\ 
		($x_2,y_2$) & \cellcolor{gray!50}($0.121,0.311$) & \cellcolor{gray!50}($0.2,0.43$) & \cellcolor{gray!50}($0.009,0.772$) & ($0,0.833$) & ($0,0.833$) \\
		($x_3,y_3$) & \cellcolor{gray!50}($0.161,0.303$) & ($0.126,0.25$) & ($0,0.428$) & ($0,0.343$) & ($0,0.325$) \\
		($x_4,y_4$) & \cellcolor{gray!50}($0.083,0.309$) & ($0,0$) & ($0,0$) & ($0,0$) & ($0,0$) \\
		$C_T(\mathbf{x},\mathbf{y})$ & $640$ & $201.062$ & $201.062$ & $169.469$ & $164.56$ \\
		$C_S(\mathbf{x},\mathbf{y})$ & $599.072$ & $199.985$ & $199.575$ & $172$ & $167.035$ \\
		$\beta_T$ & $0$ & $0.183$ & $0.210$ & $0$ & $0$ \\
		$\beta_S$ & $0$ & $0.154$ & $0.135$ & $0$ & $0$ \\
		\bottomrule
	\end{tabular}
	\vspace*{-18px}
\end{table}

Note that BNE and RBNE lead to lower latencies, with a slight difference in the simulation, compared to the presented NE configuration. By introducing altruistic autonomous vehicles, the overall cost can be further decreased.

While theoretical values conform with the fact that RBNE maximizes the robustness with the same cost as other BNE's, the simulated robustness for the RBNE case is smaller than that of the presented BNE solution. To understand this, we first note that the RBNE solution makes the free-flow road with the lowest latency (road 3) fully autonomous to be able to allocate more flow. However, as can be seen from Fig.~\ref{fig:simulation}, the mismatch between theory and simulations grows larger with increasing autonomy. This causes RBNE to be unable to allocate high prospective flow. A better maximum flow model would resolve this.

\begin{figure}
	\centering
	\vspace*{-16px}
	\includegraphics[width=\textwidth]{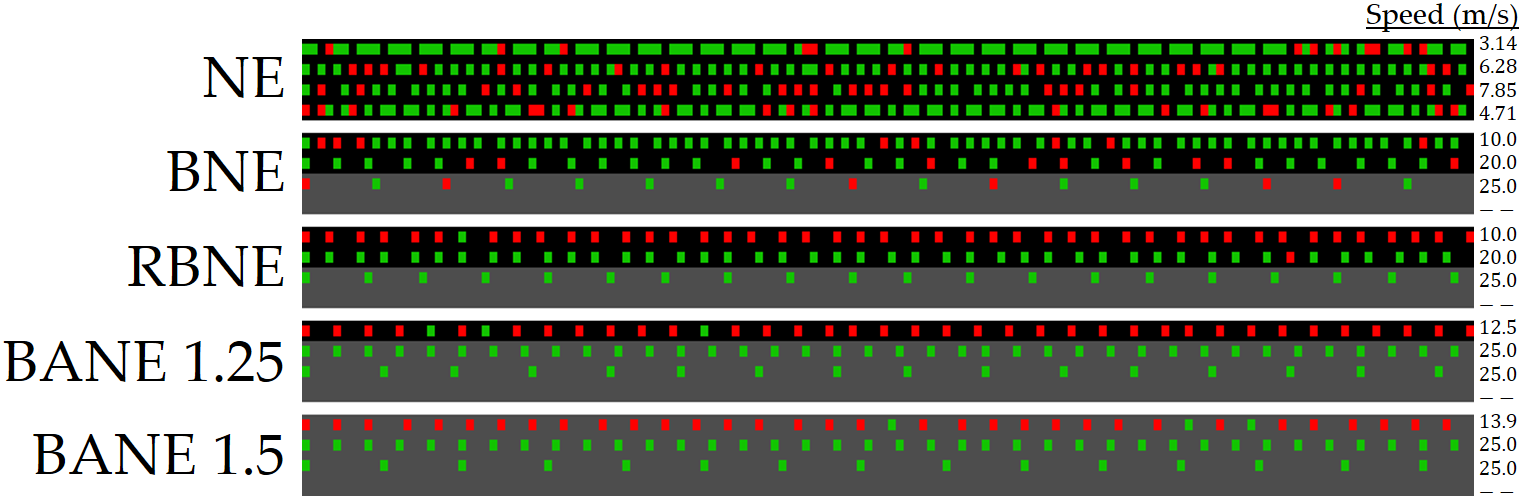}
	\vspace*{-16px}
	\caption{Roads are visualized with respect to the vehicle density values obtained from each of the solutions. Red and green dots represent regular and autonomous vehicles, respectively. We do not show headways for simplicity. Black roads are congested, while gray roads are in free-flow. On the right, average vehicle speeds (m/s) are shown.
	}
	\vspace*{-18px}
	\label{fig:road_visuals}
\end{figure}

Fig.~\ref{fig:road_visuals} shows the average speeds of vehicles on the roads and visualizes the vehicle densities for each of the theoretically presented solutions. The congested roads are black and the free-flow roads are gray. This makes clear that the NE solution leads to high latency due to heavy traffic congestion. Other solutions support the same flows with lower vehicle densities thanks to maintaining higher speeds, yielding lower latencies. Further, the RBNE makes the free-flow road fully autonomous to enable the allocation of more additional flow. Lastly, increasing the altruism level in BANE reduces the overall traffic congestion.\footnote{An animated version can be found at \url{http://youtu.be/Hy2S6zbL6Z0} with realistic numerical values for densities, headways, car lengths, speeds, etc.}
\begin{wrapfigure}{r}{0.50\textwidth}
	\centering
	\vspace*{-20px}
	\hspace*{-4pt}
	\includegraphics[width=0.51\textwidth]{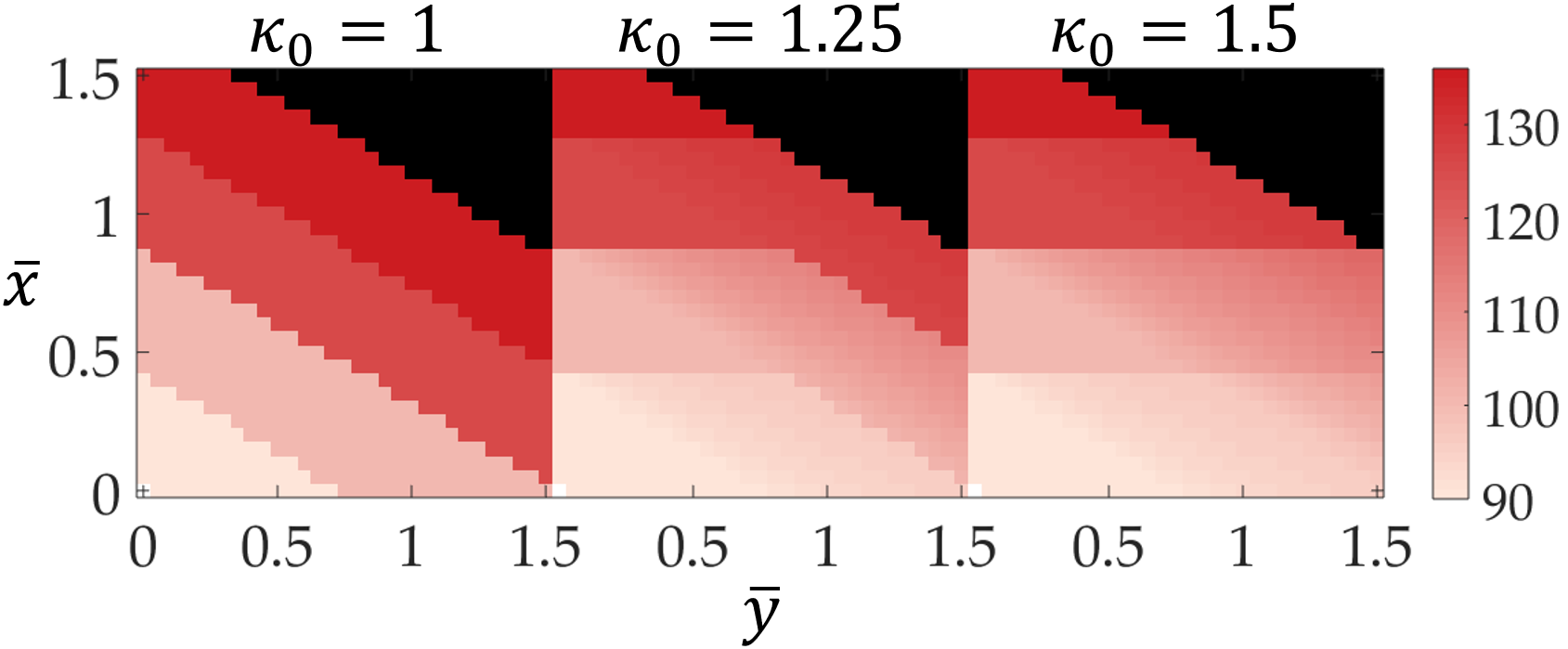}
	\vspace*{-18px}
	\caption{Average latency (seconds) obtained from BANE optimization for $\bar{x},\!\bar{y}\!\in\![0,\!1.5]$ cars per second and $\altruismUnif\!\in\!\{1,\!1.25,\!1.5\}$. Black regions represent infeasible demand.
	}
	\vspace*{-20px}
	\label{fig:heatmaps}
\end{wrapfigure}

\noindent\textbf{Effect of different altruism levels.} We also analyzed the effect of altruism level under different flow constraints. Fig.~\ref{fig:heatmaps} shows the average latency experienced by each vehicle for three different altruism levels and under different total flow requirements. We solved all cases using BANE optimization. The first heat map shows non-altruistic case where the average latencies are exactly equal to one of the free-flow latencies. This is a direct result of Theorem~\ref{thm:BNE}. In the altruistic cases, the average latency experienced by vehicles is lower and increasing the altruism level helps decrease the latency further.

\smallspacingsection{Discussion}\label{sct:discussion}
\noindent\textbf{Summary.} In this work, we develop a new model to incorporate the effect of autonomy in parallel road networks based on the fundamental diagram of traffic. We further define a notion of robustness indicating the capability of a Nash Equilibrium to be resilient to additional unforeseen flow demand. We design an optimization-based polynomial-time algorithm to find a robust Nash Equilibrium that minimizes overall latency.
We then define the concept of \emph{altruistic autonomy} to model autonomous users' willingness to accept higher latencies. We provide a polynomial-time algorithm that utilizes altruism to minimize overall latency. We demonstrate the improvements our algorithms provide using simulations.

\noindent\textbf{Limitations and Future Directions.} In this work we consider parallel networks; we wish to also study more complex transportation networks that include other modes of transportation. Congestion and altruism can be modeled for each mode of transportation. Further, in this work we use the original Krauss model \cite{krauss1997metastable} for car following; more complex and expressive models, such as Intelligent Driver Model \cite{treiber2000congested}, could yield to other interesting outcomes. This work also considers only isolated equilibria -- investigating how to move the network from one equilibrium to another is a valuable direction. Finally, our work is a first step in formalizing altruistic autonomy. We would like to perform more realistic case studies to understand emergent phenomenon in mixed-autonomy roads.

\vspace*{-12px}\section*{Acknowledgments}\vspace*{-8px}
Toyota Research Institute ("TRI") provided funds to assist the authors with their research but this article solely reflects the opinions and conclusions of its authors and not TRI or any other Toyota entity. This work was also supported by NSF grant CCF-1755808 and UC Office of the President grant LFR-18-548175.

%
%

\smallspacingsection{Supplementary Materials}

In this section, we first provide the proofs of theoretical results discussed in the main body of the paper. We then provide two case studies that illustrate properties of the set of Best-case Altruistic Nash Equilibria.
\subsection{Proof of Theoretical Results}
We first present a Lemma necessary for proving Theorem~\ref{thm:BNE}. We then provide a proof for the Lemma, then provide proofs for Theorems ~\ref{thm:BNE}, \ref{thm:BANE}, \ref{thm:compute_RBNE}, and \ref{thm:compute_BANE}.

\begin{lemma}\label{lma:free_flow}
	If the set of Nash Equilibria contains a routing with positive flow only on roads $\midcs$, then there exists a routing in the set of Nash Equilibria with positive flow only on roads $\mprimeidcs$ where $m' \le m$, and road $m'$ is in free-flow. 
\end{lemma}

\begin{proof}
	Consider a routing at Nash Equilibrium with positive flow on roads $\midcs$. Assume road $m$ is congested, as otherwise the lemma would be satisfied with $m'=m$. Note that in the congested regime, latency decreases smoothly as the volume of regular or autonomous cars increases. Accordingly, we consider transferring flow from road $m$ to the other roads in a way such that the latencies on roads ${[}m-1{]}$ stay equal to each other (and decrease together). We continue this until one of two things occurs:
	\begin{enumerate}
		\item the latencies on roads ${[}m-1{]}$ are reduced to $a_m$, the free-flow latency on road $m$, or
		\item all flow is transferred off road $m$ and the latencies on roads ${[}m-1{]}$ are greater than $a_m$.
	\end{enumerate}
	In the first case, the lemma is satisfied with $m'=m$. In the second case, we begin the process again, moving traffic from road $m-1$ to roads ${[}m-2{]}$. This continues until either we achieve case 1 above or until we are reduced to a single road. If that occurs, traffic can be routed in free-flow on that road, since any feasible flow in the congested regime is less than the maximum free-flow on a road. 
\end{proof}

\noindent\textbf{Proof of Theorem~\ref{thm:BNE}.} The definition of Nash Equilibrium and the fact that latency on a road is always greater than or equal to its free-flow latency together imply that at Nash Equilibrium, if road $m$ has positive flow then all roads with free flow latency less than $a_m$ have positive flow as well. Further, we use Lemma \ref{lma:free_flow} to show that all routings in the set of BNE will have one road in free-flow. Assume for the purposes of contradiction that we have a routing in the set of BNE in which only roads $\midcs$ have positive flow and all are congested, with equilibrium latency $\eqDelay$. The total cost is then $\eqDelay(\barx + \bary)$, where $\eqDelay > a_m$. By Lemma \ref{lma:free_flow}, another routing exists in the set of NE which uses roads $\mprimeidcs$, where $m' \le m$ and road $m'$ is in free flow. The cost of this equilibrium is $a_{m'}(\barx + \bary) \le a_m(\barx + \bary) < \eqDelay(\barx + \bary)$ contradicting our premise.
	
So far we have established that all routings in the set of BNE will have one road in free-flow and all roads with lower free-flow latency will be congested. All routings in the set of BNE will have the same free-flow road, which has the lowest free-flow latency out of all roads at free-flow in the set of NE. This is because all traffic experiences delay equal to the free-flow latency of this road and cost of a routing is strictly decreasing with the free-flow delay of this road, and we assume no two roads have the same free-flow latency. 

\noindent\textbf{Proof of Theorem~\ref{thm:BANE}.} The first property directly follows from Theorem \ref{thm:BNE}, as the regular vehicles have to be at a Nash equilibrium due to selfishness.
	
To prove the second property, we note that for roads that have higher latencies than road $\mstareq$, a Nash equilibrium is not necessary thanks to altruism. As $\ell_i(x_i,y_i,s_i)$ is a non-increasing continuous function of $y_i$ and decreasing for $s_i=1$, roads that have higher latencies than road $\mstareq$ will always be in-free flow.

Now we assume some of the roads with indices greater than $\mstareq$ and less than $\mstarall$ are in free-flow, but not at maximum flow. Then we could simply transfer some of the flow from the road $\mstarall$ to those roads and have lower overall costs, which is a contradiction. This completes the proof for the third property.

\noindent\textbf{Proof of Theorem~\ref{thm:compute_RBNE}.} Since road $\meq$, the longest equilibrium road, is in free-flow by Theorem \ref{thm:BNE} and $\mstareq$ is the minimum feasible $\meq$, our solution is restricted to the set of BNE. Then maximizing robustness ensures we find a routing in the set of RBNE. Accordingly, we can restrict our robustness maximization to routing in which the longest equilibrium road is in free-flow. This allows us to write an optimization equivalent to maximizing robustness:
\begin{align}
&\max_{\bx, \by \in \mathbb{R}^N_{\ge 0}, \gamma \ge 0} \gamma \nonumber \\
&\quad \text{s.t.} \sum_{i \in \mstareqidcs}x_i = \barx  \nonumber \\
&\qquad \; \, \sum_{i \in \mstareqidcs}y_i = \bary \nonumber \\
&\qquad \; \, \ell_i(x_i,y_i,1) = a_\mstareq \; \forall i \in {[}m-1{]} \label{eq:NE_constraint} \\
&\qquad \; \, x_\mstareq + y_\mstareq + \gamma(\barx + \bary) \le \zmax_\mstareq(x_\mstareq + \gamma \barx, y_\mstareq + \gamma \bary) \label{eq:feas_free_flow_constraint}
\end{align}

Since we are maximizing $\gamma$, constraint \eqref{eq:feas_free_flow_constraint} will be tight, allowing us to find $\gamma$ in terms of $x_\mstareq$ and $y_\mstareq$: 
\begin{align*}
	\gamma = f(x_\mstareq,y_\mstareq):=\frac{v^f_\mstareq - (h+L)x_\mstareq - (\bar{h} + L)y_\mstareq}{(h+L)\barx + (\bar{h} + L)\bary} \; ,
\end{align*}
which is affine in ${[}x_\mstareq,y_\mstareq{]}^T$. Further, for the non free-flow roads, we use \eqref{eq:NE_constraint} to find an expression for $y_i$ in terms of $x_i$ which we plug back into \eqref{eq:NE_constraint}, yielding new constraints $g_i(x_i,y_i)=0$, where $g_i$ is affine in ${[}x_i,y_i{]}^T$. This gives us the following equivalent optimization:
\begin{align*}
&\max_{\bx, \by \in \mathbb{R}^N_{\ge 0}} f(x_\mstareq,y_\mstareq) \\
&\quad \text{s.t.} \sum_{i \in \midcs}x_i = \barx \\
&\qquad \, \sum_{i \in \midcs}y_i = \bary \\
&\qquad \, g_i(x_i,y_i) = 0 \; \forall i \in {[}\mstareq-1{]} \\
&\qquad \, f(x_\mstareq,y_\mstareq) \ge 0
\end{align*}

As this is a linear program, it can be solved in $O(N^3)$ time. Finding $\mstareq$ can be done with a search in $O(N)$ time.

\noindent\textbf{Proof of Theorem~\ref{thm:compute_BANE}.} First we prove that \eqref{opt:BANE} finds a routing in the set of BANE by demonstrating the monotonicity of the cost of routing in the latency on road $\meq$. We then show that in light of this, for each choice of $\meq$ there are a maximum of $|\altruismLevSet|N$ choices of $\eqDelay$, the equilibrium latency, to consider. The associated optimal social cost of each requires $O(N^3)$ computations to evaluate.

First consider an inner maximization, which seeks to fit as much autonomous flow on roads $\meqidcs$ as possible, all experiencing latency $\eqDelay$. Let $\bx^{\text{EQ}}$ and $\by^{\text{EQ}}$ denote the elements of the regular and autonomous routings $\bx$ and $\by$ that correspond to flows on the roads $\meqidcs$, as with $\bs^\text{EQ}$. Then, solve
\begin{equation}\label{opt:max_aut_on_meq_roads}
\begin{aligned}
& \argmax_{\bx^{\text{EQ}}, \by^{\text{EQ}} \in \mathbb{R}^{|\meqidcs|},\bs \in \{0,1\}^{|\meqidcs|} }
& & \sum_{i\in \midcs}y_i
& & \\
& \qquad \quad \text{s.t.}
& & \sum_{i\in \meqidcs}y_i = \bary 
& & \\
& 
& & x_i \ge 0, y_i \ge 0 
& & \forall i \in \meqidcs \\
&
& & \ell_i(x_i,y_i,1)=\eqDelay 
& &  \forall i \in {[}\meq-1{]} \\
& 
& & \ell_\meq(x_\meq,y_\meq,s_\meq) = \eqDelay \; .
& &
\end{aligned}
\end{equation}

Using reasoning similar to that in the proof of Theorem \ref{thm:compute_RBNE}, this can be formulated as a linear program and therefore can be solved with a computational complexity $O(N^3)$ \cite{gonzaga1992path}.

We also consider an optimization which computes how many free-flow roads are required to route the autonomous traffic that does not fit on roads $\meqidcs$:
\begin{equation}\label{opt:num_ff_roads_BANE}
\begin{aligned}
& \argmin_{j\in\nidcs \backslash {[}\meq-1{]}}
& & \quad j \\
& \qquad \quad \text{s.t.}
& & \sum_{i\in \meqidcs}y_i + \sum_{i \in {[}j{]} \backslash \meqidcs}\zmax_i(0,1)\ge \bary \; ,
\end{aligned}
\end{equation}
which is an optimization that requires computations of order $O(N)$.

We temporarily restrict our attention to the case in which autonomous users have a uniform autonomy level. We wish to optimize over the following decision variables:
\begin{equation*}
\begin{aligned}
& \meq \in \nidcs
& & \text{which road is longest equilibrium road?} \\
& \mall \in \midcs \backslash {[}\meq-1{]} 
& & \text{which road is longest used road?} \\
& \eqDelay \in {[}a_\meq,a_{\meq+1}{)} 
& & \text{what is the equilibrium latency?} \\
& \bx,\by \in \mathbb{R}^n_{\ge 0}, \bs \in \{0,1\}^n 
& & \text{what is the actual routing?}
\end{aligned}
\end{equation*}

The objective function to be minimized is aggregate latency, as follows:
\begin{align}
& \eqDelay \sum_{i \in \meqidcs}(x_i+y_i) + \sum_{i\in {[}\mall-1{]} \backslash \meqidcs}a_i \zmax_i(0,1) + a_\mall(\bary-\sum_{i \in {[}\mall-1{]}}y_i) \nonumber \\ 
\text{s.t.} & (\bx^{\text{EQ}}, \by^{\text{EQ}}, \bs^{\text{EQ}}) \in \eqref{opt:max_aut_on_meq_roads} \label{eq:max_aut_flow} \\
& \mall = \eqref{opt:num_ff_roads_BANE} \label{eq:feas_longest_road} \\
& \frac{1}{\bary} \left( \sum_{i \in \meqidcs}y_i + \sum_{i \in {[}j{]} \backslash \meqidcs} \zmax_i(0,1) \right) \ge \altruismFn(\frac{a_i}{\eqDelay}) \; \forall \; j \in \nidcs \backslash {[}\meq-1{]} \label{eq:altruism}
\end{align}

Here, constraint \eqref{eq:max_aut_flow} corresponds to maximizing autonomous flows on roads $\meqidcs$ given $\ell_\meq$, \eqref{eq:feas_longest_road} corresponds to finding the minimum feasible longest used road (\emph{i.e.} optimally routing altruistic flow), and \eqref{eq:altruism} ensures that no one is more altruistic than they wish.

To solve this, recall that in the case of uniform altruism,
\begin{align*}
\altruismFn(\altruismArg) = \begin{cases}
0 & 0 \le \altruismArg \le \altruismUnif \\
1 & \altruismArg > \altruismUnif \; .
\end{cases}
\end{align*}

Further, the volume of autonomous flow that can fit on roads $\midcs$ increases with decreasing $\ell_m$. Because of this, we can restrict our search of $\eqDelay$ to critical points of the function $\delta$:
\begin{align*}
\eqDelay \in \{a_m\} \cup \{\frac{a_i}{\altruismUnif} : i\in \mallidcs \backslash {[}\meq{]}, a_\meq < \frac{a_i}{\altruismUnif} < a_{\meq+1} \} \; ,
\end{align*}
which is a set with maximum cardinality $N$. Therefore, we can find the BANE via the following algorithm:
\begin{enumerate}
	\item Enumerate through all possible values of $\meq$ ($N$ possibilities).
	\item For each possible $\meq$, enumerate through all possible values of $\eqDelay$ ($N$ possiblilities).
	\item For each combination of $\meq$ and $\eqDelay$, find the optimal routing on roads $\meqidcs$ via \eqref{opt:max_aut_on_meq_roads} (order $N^3$) and find $\mall$ and the optimal routing of autonomous vehicles on the remaining roads via \eqref{opt:num_ff_roads_BANE} (order $N$). Since these are parallel, this step requires computations of order $O(N^3)$.
\end{enumerate}
All together, this requires computations of order $O(N^5)$.

Now consider that autonomous vehicles have nonuniform altruism levels. Then,
\begin{align*}
\altruismFn(\altruismArg) = \begin{cases}
0 & 0 \le \altruismArg \le \altruismArg_0 \\
\varphi_0 & \altruismArg_0 < \altruismArg \le \altruismArg_1 \\
\varphi_1 & \altruismArg_1 < \altruismArg \le \altruismArg_2 \\
\ldots & \\
1 & \altruismArg > \altruismArg_{{|\altruismLevSet|-1}} \; .
\end{cases}
\end{align*}
We therefore must search over $\eqDelay$ in the following set:
\begin{align*}
\eqDelay \in \{a_m\} \cup \{\frac{a_i}{\altruismArg_j} : i\in \mallidcs \backslash {[}\meq{]}, \; j \in \altruismLevSet, \; a_\meq < \frac{a_i}{\altruismArg_j} < a_{\meq+1} \} \; ,
\end{align*}
which has maximum cardinality $|\altruismLevSet|N$, bringing the total computation complexity to $O(|\altruismLevSet|N^5)$.

\subsection{Case Studies}
We now provide two case studies which illustrate properties of Best-case Altruistic Nash Equilibria, specifically showing that two properties of BNE do not hold for BANE.

\subsubsection{Longest equilibrium road may be congested in BANE}\label{sct:BANE_meq_not_ff}

As shown in Theorem \ref{thm:BNE}, all routings in the set of BNE have the longest equilibrium road in free-flow, meaning that the road with largest free-flow latency that has delay equal to that experienced by human drivers is in free-flow. Counterintuitively, this property does not extend to the set of BANE. In this section we provide an example to that effect. 

\begin{figure}
	\centering
	\includegraphics[width=1\linewidth]{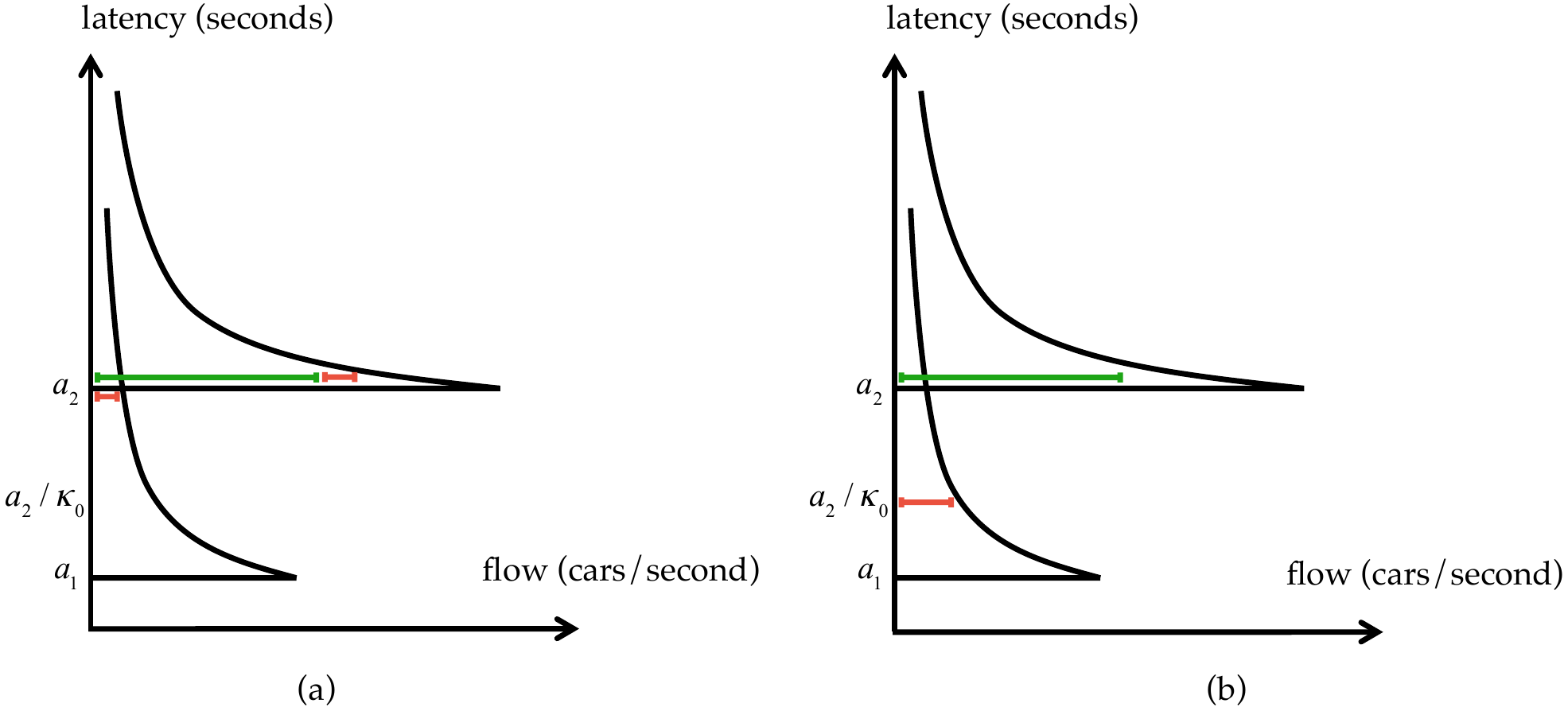}
	\caption{Illustration of a case study showing that the longest equilibrium road may be congested in the BANE. Red and green respectively denote human-driven and autonomous vehicle flow. (a) shows the only feasible routing with longest equilibrium road in free-flow, where humans are split between road 1 and road 2, and all autonomous vehicles are on road 2. (b) shows a more efficient routing, in which road 1 is now the longest equilibrium road, is congested with all human flow on it.}
	\label{fig:BANE_congested_case_study}
\end{figure}

Consider the scenario in Figure \ref{fig:BANE_congested_case_study}, which has a volume $\bary$ of autonomous vehicle flow, with uniform altruism level $\altruismUnif$ and a volume $\barx$ of human drivers. These flows are to be routed on a network of two roads, with free flow latencies $a_1$ and $a_2$ such that $\frac{a_2}{a_1}>\altruismUnif$, meaning that if road 1 is in free-flow, autonomous users will refuse to accept a delay of $a_2$. 

Let these roads have maximum flow that satisfy $\barx + \bary > \zmax_1(\barx,\bary)$ and $\barx + \bary < \zmax_2(\barx, \bary)$, so road 1 cannot accommodate all flow but road 2 can. Further, let the volume of regular traffic be such that $\ell_1(\barx,0,1)=\frac{a_2}{\epsilon_0} < a_2$. This means that if road 1 is congested with the full volume of human-driven flow, autonomous users will be willing to take road 2. However, if road 1 experiences delay less than this, autonomous users will not be willing to take road 2.

With all this in mind, any routing with the longest equilibrium road in free-flow has all traffic experiencing latency $a_2$, for a social cost of $a_2(\barx + \bary)$. No feasible routing can exist with road 1 in free-flow because if road 1 is in free-flow then autonomous users will refuse to take road 2, and road 1 cannot accommodate all flow. However, another feasible routing exists with all humans on road 1, where road 1 is congested, and all autonomous traffic on road 2. This yields a social cost of $\ell_1(\barx,0,1)\barx + a_2\bary$, which is less than the social cost of any equilibrium in which the longest equilibrium road is in free-flow.

\subsubsection{Number of equilibrium roads is not minimized in BANE}\label{sct:BANE_meq_not_min}

Theorem \ref{thm:BNE} shows that for all routings in the set of BNE have the same longest congested road and that the number of roads at equilibrium is minimized (and therefore equilibrium latency is minimized). One may conjecture that this same property applies to routings in the set of BANE, \emph{i.e.} that in the presence of altruism the social cost is minimized when limiting the equilibrium to as few roads as possible, and therefore minimizing the equilibrium latency. In this section we provide an example to disprove this conjecture. 

\begin{figure}
	\centering
	\vspace{-10pt}
	\includegraphics[width=1\linewidth]{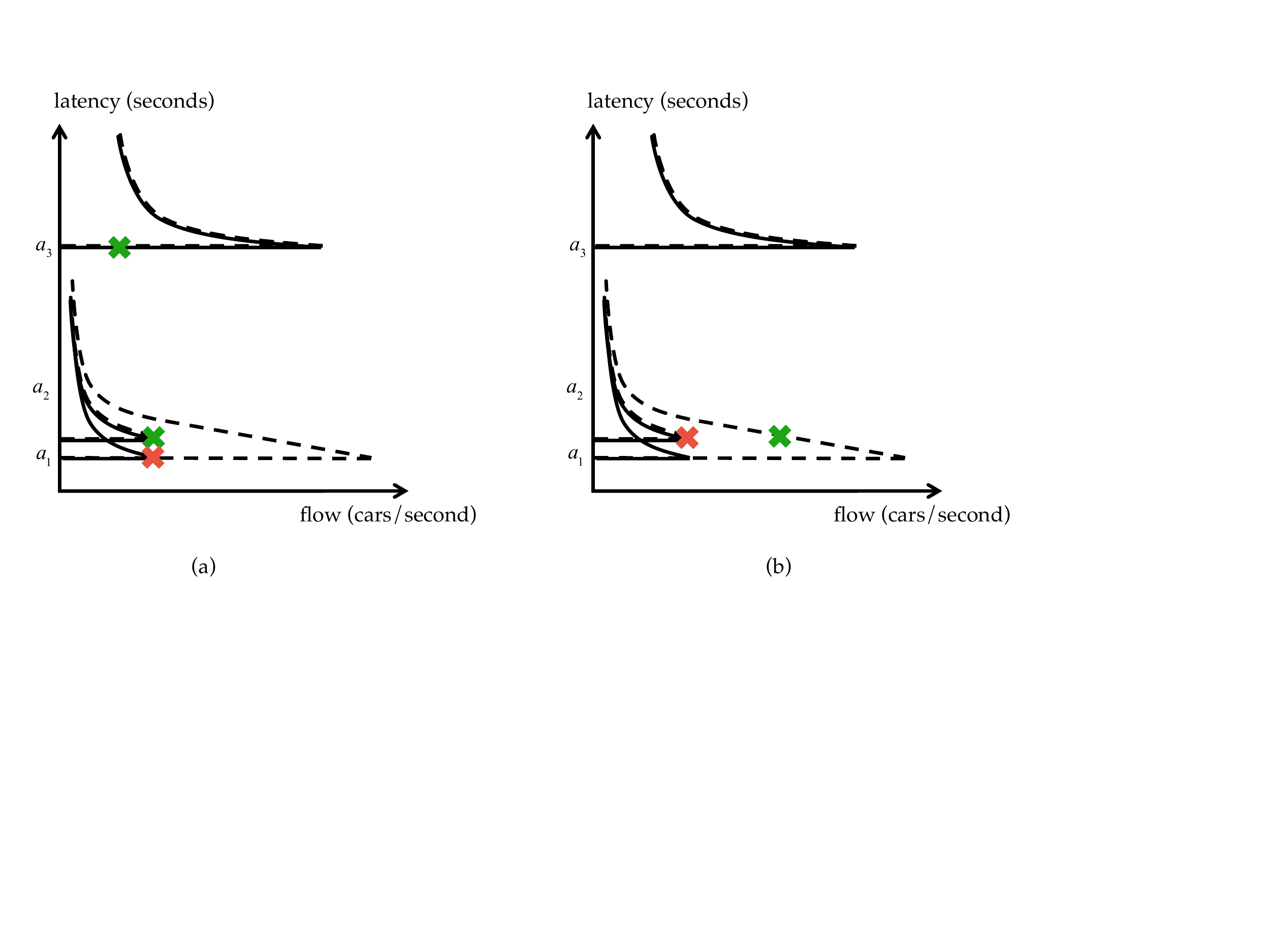}
	\vspace{-12pt}
	\caption{Illustration of a case study showing that the BANE does not necessarily minimize the number of roads at Nash Equilibrium. Red and green respectively denote human-driven and autonomous vehicle flow and solid and dashed lines represent latency functions for human-driven and autonomous flow, respectively. (a) shows that when the longest equilibrium road is road 1, all humans are on road 1, with autonomous flow split between roads 2 and 3. In (b), road 2 is the longest equilibrium road, with autonomous users on road 1 and human-driven vehicles on road 2.}
	\label{fig:not_min_eq_rds_BANE_case_study}
	\vspace{-10pt}
\end{figure}

Consider three roads with free-flow latencies $a_1<a_2<a_3$ and autonomous users with volume $\bary$ and uniform altruism $\altruismUnif>\frac{a_3}{a_1}$, and a volume of $\barx$ human-driven vehicles, with $\barx<\bary$. Let road 1 display large asymmetry in human and autonomous headways, so $\zmax_1(\barx,0) << \zmax_1(0,\bary)$, where $\barx = \zmax_1(\barx,0)$, and $\bary < \zmax_1(0,\bary)$, meaning that human-driven vehicles exactly fill road 1, but autonomous vehicles do not. Further, let $\ell_1(0,\bary,1) = a_2$, so road 1, when congested with all autonomous vehicles, has latency equal to that of road 2. Road 2 on the other hand will display no asymmetry, so $\zmax_2(\barx,0) = \zmax_2(0,\bary)$, and $\zmax_2(\barx,0) = \barx < \bary$. Finally, let road 3 have a very large maximum flow for both autonomous and human-driven vehicles so it can accommodate any flow. Figure \ref{fig:not_min_eq_rds_BANE_case_study} illustrates this scenario. 

With all this in mind, the optimal altruistic routing that minimizing the number of roads in equilibrium involves putting all humans on road 1, and having autonomous users split between roads 2 and 3 for a cost of $a_1 \barx + a_2 \zmax_2(0,\bary) + a_3 (\bary-\zmax_2(0,\bary))$. However, another feasible routing (which does not minimize number of roads in equilibrium) has all autonomous users congested on road 1 and all human drivers on road 2 for cost $a_2(\barx + \bary)$. If the autonomous users are very altruistic, delay on road 3 can be very large, so the latter routing will induce a lower social cost.

\end{document}